\documentclass[a4paper]{amsart} 
\usepackage{tikz-cd}
\usepackage{amssymb}
\usepackage{eucal}
\usepackage{ytableau}
\usepackage{spverbatim}
\usepackage{pdfsync}
\usepackage{hyperref}

\calclayout

\title{Highest weight categories and strict polynomial functors}

\date{December 22, 2015}
\dedicatory{Dedicated to the memory of Professor Sandy Green.}

\author{Henning Krause}
\address{Henning Krause\\ Fakult\"at f\"ur Mathematik\\
Universit\"at Bielefeld\\ D-33501 Bielefeld\\ Germany.}
\email{hkrause@math.uni-bielefeld.de}

\newtheorem{lem}{Lemma}[section]
\newtheorem{prop}[lem]{Proposition}
\newtheorem{cor}[lem]{Corollary}
\newtheorem{thm}[lem]{Theorem}

\theoremstyle{remark}
\newtheorem{rem}[lem]{Remark}

\theoremstyle{definition}
\newtheorem{exm}[lem]{Example}
\newtheorem{defn}[lem]{Definition}

\numberwithin{equation}{section}

\newcommand{\smatrix}[1]{\left[\begin{smallmatrix}#1\end{smallmatrix}\right]}

\renewcommand{\mod}{\operatorname{mod}\nolimits}

\newcommand{\diag}{\operatorname{diag}\nolimits}
\newcommand{\card}{\operatorname{card}\nolimits}

\newcommand{\proj}{\operatorname{proj}\nolimits}
\newcommand{\inj}{\operatorname{inj}\nolimits}

\newcommand{\add}{\operatorname{add}\nolimits}

\newcommand{\free}{\operatorname{free}\nolimits}
\newcommand{\id}{\operatorname{id}\nolimits}

\newcommand{\Mod}{\operatorname{Mod}\nolimits}
\newcommand{\rep}{\operatorname{rep}\nolimits}
\newcommand{\Rep}{\operatorname{Rep}\nolimits}

\newcommand{\End}{\operatorname{End}\nolimits}
\newcommand{\Fun}{\operatorname{Fun}\nolimits}
\newcommand{\Hom}{\operatorname{Hom}\nolimits}

\newcommand{\RHom}{\operatorname{\mathbf{R}Hom}\nolimits}

\newcommand{\HOM}{\operatorname{\mathcal H \;\!\!\mathit o \mathit m}\nolimits}

\renewcommand{\Im}{\operatorname{Im}\nolimits}
\newcommand{\Ker}{\operatorname{Ker}\nolimits}

\newcommand{\Ext}{\operatorname{Ext}\nolimits}
\newcommand{\Tor}{\operatorname{Tor}\nolimits}
\newcommand{\Filt}{\operatorname{Filt}\nolimits}

\newcommand{\sgn}{\operatorname{sgn}\nolimits}

\newcommand{\Lotimes}{\otimes^{\mathbf{L}}}

\newcommand{\rej}{\operatorname{rej}\nolimits}
\newcommand{\tr}{\operatorname{tr}\nolimits}

\newcommand{\Ab}{\mathrm{Ab}}
\newcommand{\op}{\mathrm{op}}

\newcommand{\perf}{\mathrm{perf}}

\newcommand{\lto}{\longrightarrow}
\newcommand{\xto}{\xrightarrow}
\newcommand{\lotimes}{\otimes^{\mathbf L}}

\def\a{\alpha}
\def\b{\beta}

\def\g{\gamma}
\def\p{\phi}

\def\s{\sigma}

\def\la{\lambda}
\def\De{\Delta}
\def\Ga{\Gamma}
\def\La{\Lambda}

\def\A{{\mathcal A}}
\def\B{{\mathcal B}}

\def\M{{\mathcal M}}

\def\P{{\mathcal P}}

\def\bfD{{\mathbf D}}

\def\fra{{\mathfrak a}}

\def\frS{{\mathfrak S}}

\begin{document}

\begin{abstract}
  Highest weight categories are described in terms of standard objects
  and recollements of abelian categories, working over an arbitrary
  commutative base ring. Then the highest weight structure for
  categories of strict polynomial functors is explained, using the
  theory of Schur and Weyl functors.  A consequence is the well-known
  fact that Schur algebras are quasi-hereditary.
\end{abstract}

\maketitle
\setcounter{tocdepth}{1}
\tableofcontents

\section*{Introduction}

Highest weight categories and quasi-hereditary algebras were
introduced in a series of papers by Cline, Parshall, and Scott
\cite{CPS1988, PS1988,Sc1987}. A classical example are polynomial
representations of general linear groups which are equivalent to
modules over Schur algebras \cite{Gr1980}, and these are well-known to
be quasi-hereditary algebras \cite{DEP1980,Do1987}.  These notes
present an alternative approach to this subject, and a distinctive
feature is that we work over an arbitrary commutative base ring.

The first part is devoted to giving descriptions of highest weight
categories in terms of standard objects (Theorem~\ref{th:standard})
and recollements of abelian categories (Theorem~\ref{th:khwt}).  Also,
we discuss Ringel duality which is based on the notion of a
characteristic tilting object \cite{Ri1991}, and we establish a
precise connection with Serre duality (Theorem~\ref{th:ringel-dual}).

The filtration of a highest weight category via recollements of
abelian categories induces a filtration of the corresponding derived
category via recollements of triangulated categories
(Proposition~\ref{pr:kgldim}). Such filtrations of derived categories
provide a somewhat characteristic property of highest weight
categories and have been studied extensively \cite{Ka2015}.

For another interesting approach towards highest weight categories via
$A_\infty$-categories and bocses we refer to recent work in
\cite{KKO2014,Ku2015}. 

In the second part of these notes we explain the highest weight
structure for categories of strict polynomial functors \cite{FS1997},
working over an arbitrary commutative ring and using some of the
principal results from the theory of Schur and Weyl functors
\cite{ABW1982}.  The essential ingredients of the highest weight
structure are:
\begin{enumerate}
\item[--] The Weyl functors are precisely the standard objects
  (Theorem~\ref{th:weight}).
\item[--] The Cauchy decomposition provides a filtration of any
  projective object whose associated graded object is a direct sum
  of Weyl functors (Corollary~\ref{co:cauchy-lambda}).
\item[--] The exterior powers provide a characteristic tilting object and
  the category of strict polynomial functors is Ringel self-dual
  (Theorem~\ref{th:char-tilting}).
\end{enumerate}

The material in the second part is elementary, based to a large extent
on classical facts from multilinear algebra. The language of strict
polynomial functors is employed because of its flexibility. Evaluating
strict polynomial functors at a free module of finite rank makes it
easy to transfer this work to the representation theory of Schur
algebras; for an explicit discussion we recommend Hashimoto's notes
\cite{Ha2012}.

\part{Highest Weight categories}

\section{$k$-linear highest weight categories}\label{se:khwt}

Highest weight categories were introduced by Cline, Parshall, and
Scott \cite{CPS1988,PS1988}. The original definition is formulated in
the setting of abelian length categories. Versions for $k$-linear
exact categories have been considered by various authors
\cite{DS1994,Ef2014,Ro2008}.

From now on we fix a commutative ring $k$. We consider additive
categories that are \emph{$k$-linear}. This means the morphisms sets
are $k$-modules, and the composition maps are $k$-linear.

For a ring $\La$, let $\Mod\La$ denote the category of (right)
$\La$-modules and let $\proj\La$ denote the full subcategory of
finitely generated projective $\La$-modules.

\subsection*{$k$-linear exact categories}

Let $\A$ be an exact category \cite{Qu1973}. Thus $\A$ is an additive
category and the exact structure of $\A$ is given by a distinguished
class of \emph{exact sequences} $0\to X\to Y\to Z\to 0$ in $\A$ which
are kernel-cokernel pairs and sometimes called \emph{admissible}.

We recall from \cite[Appendix~A]{Ke1990} a useful
construction. Suppose that $\A$ is essentially small and let
$\widehat\A$ denote the category of left exact functors $\A^\op\to\Ab$
into the category of abelian groups. Then $\widehat\A$ is a
Grothendieck abelian category and the Yoneda functor
\[\A\lto \widehat\A,\quad X\mapsto h_X=\Hom_\A(-,X)\]
is fully faithful and exact, inducing a bijection
\[\Ext_\A^1(X,Y)\xto{\sim}\Ext^1_{\widehat\A}(h_X,h_Y).\]
Note that any exact functor $f\colon\A\to\B$ extends uniquely to an
exact and colimit preserving functor $\widehat\A\to\widehat\B$ (the
left adjoint of the functor $\widehat\B\to\widehat\A$ given by
precomposition with $f$).

Recall that an object $P$ in $\A$ is \emph{projective} if for every
exact sequence $0\to X\to Y\to Z\to 0$ in $\A$ the induced map
$\Hom_\A(P,Y)\to \Hom_\A(P,Z)$ is surjective. The object $P$ is a
\emph{generator} if for every object $X$ there is an exact sequence
$0\to N\to P^r\to X\to 0$ for some positive integer $r$.

\begin{lem}\label{le:exact}
  Suppose that $\A$ admits a projective generator $P$ and set
  $\La=\End_\A(P)$.  Then evaluation at $P$ induces an
  equivalence $\widehat\A\xto{\sim}\Mod\La$.

  Conversely, if $\widehat\A$ is equivalent to $\Mod\Ga$ for some ring
  $\Ga$, then the equivalence identifies $\Ga$ with a projective
  generator of $\A$.
\end{lem}
\begin{proof}
  Sending a $\La$-module $M$ to $\Hom_\La(\Hom_\A(P,-),M)$ gives a
  quasi-inverse $\Mod\La\to\widehat\A$. For the converse observe that
  any functor in $\widehat \A$ is the epimorphic image of a direct sum
  of representable functors. Thus $\Ga$ identifies with a direct
  summand of a finite direct sum of representables.
\end{proof}

For a $k$-algebra $\La$ we denote by $\mod(\La,k)$ the category of
$\La$-modules that are finitely generated projective over $k$. This is
a $k$-linear exact category where a sequence is by definition exact if
it is split exact when restricted to the category of $k$-modules. Note
that $\Hom_k(-,k)$ induces a $k$-linear equivalence
\[\mod(\La,k)^\op\xto{\sim}\mod(\La^\op,k).\]

Suppose that $\A$ admits a projective generator $P$ and set
$\La=\End_\A(P)$. If $\Hom_\A(P,X)$ is finitely generated projective
over $k$ for all $X$ in $\A$, then $\Hom_\A(P,-)$ and evaluation at
$P$ make the following diagram commutative.
\[
\begin{tikzcd}
\A \arrow[tail]{d}\arrow{rrr}{\Hom_\A(P,-)}&&&\mod(\La,k)\arrow[tail]{d}\\
\widehat\A \arrow{rrr}{\sim}&&&\Mod\La
\end{tikzcd}
\]
All functors are fully faithful and exact, but the top one need not be an equivalence.

\subsection*{Recollements}

We recall the definition of a recollement using the standard notation
\cite[1.4]{BBD1982}. 

\begin{defn}
  A \emph{recollement} of abelian (triangulated) categories is a
  diagram of functors
\begin{equation}\label{eq:rec}
\begin{tikzcd}
  \A' \arrow[tail]{rr}[description]{i_*=i_!} &&\A
  \arrow[twoheadrightarrow,yshift=-1.5ex]{ll}{i^!}
  \arrow[twoheadrightarrow,yshift=1.5ex]{ll}[swap]{i^*}
  \arrow[twoheadrightarrow]{rr}[description]{j^!=j^*} &&\A''
  \arrow[tail,yshift=-1.5ex]{ll}{j_*}
  \arrow[tail,yshift=1.5ex]{ll}[swap]{j_!}
\end{tikzcd}
\end{equation}
satisfying the following conditions:
\begin{enumerate}
\item  $i_*$ and $j^*$ are exact functors of  abelian (triangulated) categories.
\item $(i^*,i_*)$, $(j^*,j_*)$, $(i_!,i^!)$, and $(j_!,j^!)$  are adjoint pairs.
\item $i_*$, $j_*$, and $j_!$ are fully faithful functors.
\item An object in $\A$ is annihilated by $j^*$ iff it
  is in the essential image of $i_*$.
\end{enumerate}
The recollement is called \emph{homological} if the functor $i_*$
induces for all $X,Y\in\A'$ and $p\ge 0$ isomorphisms
\[\Ext^p_{\A'}(X,Y)\xto{\sim}\Ext^p_\A(i_*(X),i_*(Y)).\]
\end{defn}

A diagram \eqref{eq:rec} without the left adjoints $i^*$ and $j_!$
that satisfies all the relevant conditions of a recollement is called
\emph{localisation sequence}. Analogously, a diagram \eqref{eq:rec}
without the right adjoints $i^!$ and $j_*$ that satisfies all the
relevant conditions of a recollement is called \emph{colocalisation
  sequence}.

Given a recollement \eqref{eq:rec} and an object $X$ in
$\A$, we have the counit $j_!j^!(X)\to X$ and the unit
$X\to i_*i^*(X)$.  These fit into an exact sequence
\[j_!j^!(X)\lto X\lto i_*i^*(X)\lto 0\qquad \text{($\A$ abelian)}\]
and an exact triangle
  \[j_!j^!(X)\lto X\lto i_*i^*(X)\lto \qquad \text{($\A$ triangulated)}.\]

\subsection*{$k$-linear highest weight categories}

We give the definition of a highest weight category relative to a base
ring $k$, following closely Rouquier \cite{Ro2008}. We assume that the
set of weights is finite and totally ordered, leaving the
generalisation to locally finite posets to the interested reader.

\begin{defn}\label{de:khwt}
  Let $\A$ be a $k$-linear exact category. Suppose that
  $\A\xto{\sim}\mod(\La,k)$ for some $k$-algebra $\La$ that is
  finitely generated projective over $k$. Then $\A$ is called
  \emph{$k$-linear highest weight category} if there are finitely many
  exact sequences
\begin{equation}\label{eq:khwt}
0\lto U_i\lto P_i\lto \De_i \lto 0\qquad (1\le i\le n)
\end{equation}
  in $\A$ satisfying the following:
\begin{enumerate}
\item $\End_\A(\De_i)\cong k$ for all $i$.
\item $\Hom_\A(\De_i,\De_j)=0$ for all $i> j$. 
\item $U_i$ belongs to $\Filt(\De_{i+1},\ldots,\De_n)$ for all $i$.
\item $\bigoplus_{i=1}^n P_i$ is a projective generator of $\A$.
\end{enumerate}
The objects $\De_1,\ldots,\De_n$ are called \emph{standard objects}. 
\end{defn}

Note that the sequence \eqref{eq:khwt} implies 
\[\Ext^1_\A(\De_i,\De_j)=0\qquad\text{for all}\qquad  i\ge j.\]

The structure of a highest weight category is determined by the
ordered set of standard objects; see Theorem~\ref{th:standard}. Thus
an \emph{equivalence} of highest weight categories is an equivalence
of categories which preserves standard objects and their ordering.

Following \cite{CPS1990} we call a $k$-algebra \emph{split
  quasi-hereditary} if it is the endomorphism ring of a projective
generator of a $k$-linear highest weight category. Later we will see
that the standard objects $\De_1,\ldots,\De_n$ in $\mod(\La,k)$ give
rise to a sequence of surjective algebra homomorphisms
\[\La=\La_n\to\La_{n-1}\to\cdots\to\La_1\to\La_0=0\]
which makes it possible to define split quasi-hereditary algebras
in terms of ideal chains.

\subsection*{Standardisation}
We give an axiomatic description of the standard objects of a highest
weight category.  

Let $\A$ be an exact category and fix a sequence of objects
$\De_1,\ldots,\De_n$. We consider the following conditions:
\begin{enumerate}
\item[($\De$1)] $\Hom_\A(\De_i,\De_j)=0$ for all $i>j$. 
\item[($\De$2)] $\Ext^1_\A(\De_i,\De_j)=0$ for all $i\ge j$. 
\item[($\De$3)] $\Ext^1_\A(X,\De_j)$ is finitely generated over
  $\End_\A(\De_j)^\op$ for all $X\in\A$.
\end{enumerate}

For a ring $\La$ let $\free \La$ denote the category of free
$\La$-modules of finite rank.

\begin{lem}\label{le:standard-coloc}
Suppose that \emph{($\De$1)--($\De$2)} hold and set  $\Ga_t=\End_\A(\De_t)$
for $1\le t\le n$. Then the functor
  $\Hom_\A(\De_t,-)$ induces a colocalisation sequence
\begin{equation}\label{eq:standard-coloc}
\begin{tikzcd}
  \Filt(\De_1,\ldots,\De_{t-1})\arrow[tail,yshift=-0.75ex]{rr}[swap]{i_*}
  &&\Filt(\De_1,\ldots,\De_t)
  \arrow[twoheadrightarrow,yshift=0.75ex]{ll}[swap]{i^*}
  \arrow[twoheadrightarrow,yshift=-0.75ex]{rr}[swap]{j^!} &&\free \Ga_t
  \arrow[tail,yshift=0.75ex]{ll}[swap]{j_!}
\end{tikzcd}
\end{equation}
and all functors are exact. Each $X$ in $\Filt(\De_1,\ldots,\De_t)$
fits into an exact sequence
\[0\lto j_!j^!(X)\lto X\lto i_*i^*(X)\lto 0.\]
\end{lem}
\begin{proof}
  The object $\De_t$ is projective in $\Filt(\De_1,\ldots,\De_t) $.
  An induction on the length of a filtration of an object $X$ in
  $\Filt(\De_1,\ldots,\De_t)$ yields some $r\ge 0$ and an exact
  sequence $0\to X''\to X\to X'\to 0$ with $X'$ in
  $\Filt(\De_1,\ldots,\De_{t-1})$ and $X''\cong \De_t^r$. To see this,
  let $0\to Y\to X\to\De_j\to 0$ be an exact sequence in
  $\Filt(\De_1,\ldots,\De_t) $. The assertion for $X$ follows from
  that for $Y$. This is immediate If $j<t$ because we take $X''=Y''$.
  Otherwise, $X\cong Y\oplus\De_t$ and we can add
  $0\to\De_t\xto{\id}\De_t\to 0\to 0$ to the exact sequence for $Y$.

  Now set $i^*(X)=X'$. Also, set $j^!(X)=\Hom_\A(\De_t,X)$ and
  $j_!(\Ga_t^r)=\De_t^r$.  This gives $i_*i^*(X)\cong X'$ and
  $j_!j^!(X)\cong X''$. The exactness is obvious for the functors
  $i_*$, $j^!$, and $j_!$.  For $i^*$ it follows from the snake lemma.
\end{proof}

\begin{lem}\label{le:standard}
Suppose that \emph{($\De$2)--($\De$3)} hold.
Then there are exact sequences
\begin{equation*}
0\lto U_t\lto P_t\lto \De_t \lto 0\qquad (1\le t\le n)
\end{equation*}
in $\A$ such that $U_t$ belongs to $\Filt(\De_{t+1},\ldots,\De_n)$ for
all $t$ and $\bigoplus_{t=1}^n P_t$ is a projective generator of
$\Filt(\De_1,\ldots,\De_n)$.
\end{lem}
\begin{proof}
See \cite[Lemma~5.8]{Kr2015}.
\end{proof}

The following result characterises the standard objects of a
$k$-linear highest weight category and is an analogue of a result of
Dlab and Ringel \cite[Theorem~2]{DR1992}. This gives rise to an
alternative definition of highest weight categories which does not
involve the choice of exact sequences.

\begin{thm}\label{th:standard}
  Let $\A$ be a $k$-linear exact category and assume that
  $\Ext^1_\A(X,Y)$ is finitely generated over $k$ for all
  $X,Y\in\A$. Then a sequence of objects $\De_1,\ldots,\De_n$ in $\A$
  identifies with the standard objects $\De'_1,\ldots,\De'_n$ of a
  $k$-linear highest weight category $\A'$ via an exact equivalence
\[\A\supseteq \Filt(\De_{1},\ldots,\De_n)\xto{\sim}\Filt(\De'_{1},\ldots,\De'_n)\subseteq\A'\]
if and only if the following holds:
\begin{enumerate}
\item $\End_\A(\De_i)\cong k$ for all $i$.
\item $\Hom_\A(\De_i,\De_j)=0$ for all $i> j$. 
\item $\Ext^1_\A(\De_i,\De_j)=0$ for all $i\ge  j$. 
\item The endomorphism ring of a projective generator of
  $\Filt(\De_{1},\ldots,\De_n)$ is finitely generated projective over
  $k$.
\end{enumerate}
\end{thm}
\begin{proof}
  Clearly, the standard objects of a $k$-linear highest weight
  category satisfy the above properties. In order to show the converse
  choose any projective generator $P$ of $\Filt(\De_1,\ldots,\De_n)$
  which exists by Lemma~\ref{le:standard}.  We claim that
  $\Hom_\A(P,X)$ is finitely generated projective over $k$ for all $X$
  in $\Filt(\De_1,\ldots,\De_n)$.  Then the assertion of the theorem
  follows because we can choose $\A'=\mod(\La,k)$ for
  $\La=\End_\A(P)$, thanks to Lemma~\ref{le:standard}.

  The claim is shown by induction on $n$. We use the colocalisation
  sequence \eqref{eq:standard-coloc} for $i=n$. Given $X$ in
  $\Filt(\De_1,\ldots,\De_{n})$ set $X'=i_*i^*(X)$ and
  $X''=j_!j^!(X)$. Note that $P'$ is a projective generator of
  $\Filt(\De_1,\ldots,\De_{n-1})$.  The claim follows from the exact sequence
  \[0\to\Hom_\A(P,X'')\to\Hom_\A(P,X)\to\Hom_\A(P,X')\to 0.\]
  The $k$-module $\Hom_\A(P,X')\cong \Hom_\A(P',X')$ is finitely
  generated projective by induction, and (4) implies that
  $\Hom_\A(P,X'')$ is finitely generated projective.
\end{proof}

\subsection*{Recollements}

The following result characterises $k$-linear highest weight
categories in terms of recollements; it is the analogue of a result
for abelian length categories \cite{Kr2015,PS1988}.  We need to
involve the completion $\widehat\A$ of an exact category $\A$, because
there is in general no reason for the existence of recollements on the
level of subcategories of $\A$.

\begin{thm}\label{th:khwt}
  Let $\A$ be a $k$-linear exact category.  Suppose that
  $\A\xto{\sim}\mod(\La,k)$ for some $k$-algebra $\La$ that is
  finitely generated projective over $k$. Then the following are
  equivalent:
\begin{enumerate}
\item The category $\A$ is a $k$-linear highest weight category.
\item There is a finite chain of full exact
  subcategories\[0=\A_0\subseteq \A_1\subseteq
  \ldots\subseteq\A_n=\A\]
  such that each inclusion $\A_{i-1}\to\A_i$ induces a homological
  recollement of abelian categories
\[
\begin{tikzcd}
  \widehat{\A_{i-1}} \arrow[tail]{rr}&&\widehat{\A_i}
  \arrow[twoheadrightarrow,yshift=-1.5ex]{ll}
  \arrow[twoheadrightarrow,yshift=1.5ex]{ll}
  \arrow[twoheadrightarrow]{rr} &&\Mod k
  \arrow[tail,yshift=-1.5ex]{ll}
  \arrow[tail,yshift=1.5ex]{ll}
\end{tikzcd}
\quad\text{with}\quad \A_{i-1}= \widehat{\A_{i-1}} \cap\A_i.
\]
\end{enumerate}
\end{thm}

The proof of Theorem~\ref{th:khwt} provides for each $1\le i\le n$ a
$k$-algebra $\La_i$ such that $\A_i\xto{\sim}\mod(\La_i,k)$ and a
surjective algebra homomorphism $\La_{i}\to\La_{i-1}$.
 
\begin{proof}
  Fix a projective generator $P$ of $\A$ and set $\La=\End_\A(P)$. We
  identify $\A\xto{\sim}\mod(\La,k)$ via $\Hom_\A(P,-)$ and
  $\widehat\A\xto{\sim}\Mod\La$ via evaluation at $P$; see
  Lemma~\ref{le:exact}.

  (1) $\Rightarrow$ (2): Suppose that $\A$ is a $k$-linear highest
  weight category satisfying the conditions in
  Definition~\ref{de:khwt}, and we may assume $P=\bigoplus_iP_i$. We
  give a recursive construction of a chain
 \[0=\A_0\subseteq \A_1\subseteq \ldots\subseteq\A_n=\A\]
 of full subcategories satisfying the conditions in (2). To this end
 consider the colocalisation sequence \eqref{eq:standard-coloc} for
 $i=n$. The left adjoint
 \[\Filt(\De_1,\ldots,\De_{n})\lto \Filt(\De_1,\ldots,\De_{n-1})\] takes
 the object $P$ to a projective generator $\bar P$ of
 $\Filt(\De_1,\ldots,\De_{n-1})$. We set
 \[\A_{n-1}=\{X\in\A\mid\Hom_\A(\De_n,X)=0\}
 \qquad\text{and}\qquad \La_{n-1}=\End_\A(\bar P).\]
 Note that $\Hom_\A(\bar P,X)\cong\Hom_\A(P,X)$ is finitely generated
 projective over $k$ for all $X$ in $\A_{n-1}$. Also, it is easily
 checked that $\bar P$ is a projective generator of $\A_{n-1}$. Thus
 $\Hom_\A(\bar P,-)$ yields an equivalence
 $\A_{n-1}\xto{\sim}\mod(\La_{n-1},k)$. It follows from
 Theorem~\ref{th:standard} that $\A_{n-1}$ is a highest weight
 category with standard objects $\De_1,\ldots,\De_{n-1}$. We have
 $\widehat{\A_{n-1}}\xto{\sim}\Mod\La_{n-1}$ by Lemma~\ref{le:exact},
 and the functor $\Hom_\La(\De_n,-)\colon\Mod\La\to \Mod k$ induces
 the following
 recollement.
\begin{equation}\label{eq:khwt-rec}
\begin{tikzcd}
  \Mod\La_{n-1} \arrow[tail]{rr}&&\Mod\La
  \arrow[twoheadrightarrow,yshift=-1.5ex]{ll}
  \arrow[twoheadrightarrow,yshift=1.5ex]{ll}
  \arrow[twoheadrightarrow]{rr} &&\Mod k
  \arrow[tail,yshift=-1.5ex]{ll}
  \arrow[tail,yshift=1.5ex]{ll}
\end{tikzcd}
\end{equation}
This recollement is homological by \cite[Proposition~A.1]{Kr2015},
because for each projective object $X$ in $\Mod\La$ the counit
$ j_!j^!(X)\to X$ is a monomorphism by Lemma~\ref{le:standard-coloc}.

(2) $\Rightarrow$ (1): Fix a chain of full subcategories
$\A_i\subseteq\A$ satisfying the conditions in (2).  We show by
induction on $n$ that $\A$ is a highest weight category.  Let $\De_n$
denote the image of $k$ under the left adjoint $\Mod k\to\Mod\La$.
Clearly, $\End_\La(\De_n)\cong k$ and $\De_n$ is a finitely generated
projective $\La$-module that belongs therefore to $\A$.  The inclusion
$\widehat{\A_{n-1}}\to\Mod\La$ identifies $\widehat{\A_{n-1}}$ with
$\Mod\La/\fra$ for some idempotent ideal $\fra$; see
\cite[Proposition~7.1]{Au1974}. More precisely, the left adjoint of
$\widehat{\A_{n-1}}\to\Mod\La$ sends $\La$ to $\La/\fra$ which is a
projective generator of $\A_{n-1}$ by Lemma~\ref{le:exact}. In
particular, $\La_{n-1}=\La/\fra$ is finitely generated projective over
$k$ and $\A_{n-1}=\mod(\La_{n-1},k)$.  The induction hypothesis for
$\A_{n-1}$ yields a collection of exact sequences
\[0\lto \bar U_i\lto \bar P_i\lto \De_i \lto 0\qquad (1\le i < n)\]
and we modify them as follows to obtain exact sequences
\eqref{eq:khwt}. Using the fact that extension groups of objects in
$\A$ are finitely generated over $k$, we can form the universal
extension
\[0\lto \De_n^r\lto P_i\lto\bar P_i\lto 0\]
in $\A$; that is, the induced map
$\Hom_\A(\De_n^r,\De_n)\to\Ext_\A^1(\bar P_i,\De_n)$ is surjective. A
standard argument (as in the proof of Lemma~\ref{le:standard}) shows
that $P_i$ is a projective object. Forming the pull-back diagram
\begin{equation*}
\begin{tikzcd}
{}&0\arrow{d}&0\arrow{d}\\
{}&\De_n^r\arrow{d}\arrow[equal]{r}&\De_n^r\arrow{d}\\
0\arrow{r}&U_i\arrow{d}\arrow{r}&P_i\arrow{d}\arrow{r}&\De_i\arrow[equal]{d}\arrow{r}&0\\
0\arrow{r}&\bar U_i\arrow{d}\arrow{r}&\bar P_i\arrow{d}\arrow{r}&\De_i\arrow{r}&0\\
{}&0&0
\end{tikzcd}
\end{equation*}
we get exact sequences
\eqref{eq:khwt} with $U_i$ in $\Filt(\De_{i+1},\ldots,\De_n)$, where
$P_n:=\De_n$ and $U_n:=0$. It remains to observe that
$\bigoplus_i P_i$ is a projective generator of $\A$.
\end{proof}

\subsection*{Properties of $k$-linear highest weight categories}

We formulate some consequences of Theorem~\ref{th:khwt}.  To this end
fix a $k$-linear highest weight category $\A=\mod(\La,k)$ with chain
of subcategories
\[0=\A_0\subseteq\A_1\subseteq\ldots\subseteq\A_n=\A\qquad\text{and}\qquad
\A_i=\mod(\La_i,k).
\]

For each $1\le i\le n$ we identify $\End_\A(\De_i)=k$. Then the
functor
\[\A_i\lto\proj k,\quad X\mapsto \Hom_{\La_i}(\De_i,X)\cong 
X\otimes_{\La_i}\Hom_{\La_i}(\De_i,\La_i)\]
admits the left adjoint $-\otimes_k\De_i$ and the right adjoint
$\Hom_k(\Hom_{\La_i}(\De_i,\La_i),-)$.  This yields the following
diagram of exact functors
\begin{equation}\label{eq:rec-khwt}
\begin{tikzcd}
\A_{i-1} \arrow[tail]{rr}&&\A_i
  \arrow[twoheadrightarrow]{rr} &&\proj k
  \arrow[tail,yshift=-1.5ex]{ll}
  \arrow[tail,yshift=1.5ex]{ll}
\end{tikzcd}
\end{equation}
which one may think of as an incomplete recollement.

The standard object $\De_i$ equals the image of $k$ under the
left adjoint $\proj k\to\A_i$ while the \emph{costandard object}
$\nabla_i$ is by definition the image of $k$ under the right adjoint
$\proj k\to\A_i$. In particular we have
\begin{equation*}
\Hom_\A(\nabla_j,\nabla_i)=0\quad\text{($i> j$)}\qquad\text{and}\qquad
\Ext^1_\A(\nabla_j,\nabla_i)=0\quad\text{($i\ge j$)}.
\end{equation*}

\begin{prop}
  Let $\A$ be a $k$-linear highest weight category. Then the category
  $\A^\op$ is a $k$-linear highest weight category.
\end{prop}
\begin{proof}
  We identify $\A^\op=\mod(\La^\op,k)$ and use the duality
  $\Hom_k(-,k)$. Set $\De'_i=\Hom_k(\nabla_i,k)$ for $1\le i\le n$.
  Using Theorem~\ref{th:standard} it is easily checked that
  $\A^\op$ is a $k$-linear highest weight category with standard
  objects $\De'_1,\ldots,\De'_n$.
\end{proof}

We observe that the duality $\Hom_k(-,k)$ induces an equivalence
\begin{equation}\label{eq:duality}
\mod(\La,k)^\op\supseteq\Filt(\De_1,\ldots,\De_n)^\op
\xto{\sim}\Filt(\nabla_1,\ldots,\nabla_n)\subseteq\mod(\La^\op,k)
\end{equation}
which maps each $\De_i$ to $\nabla_i$.

For a ring $\La$ we write $\bfD^\perf(\La)=\bfD^b(\proj\La)$ for the
category of perfect complexes over $\La$.

\begin{prop}\label{pr:kgldim}
  Let $\A$ be a $k$-linear highest weight category and let $\La$
  denote the endomorphism ring of a projective generator. Then we have
  a triangle equivalence $\bfD^\perf(\La)\xto{\sim}\bfD^b(\A)$ and
  each inclusion $\A_{t-1}\to\A_t$ induces a recollement of
  triangulated categories
\[
\begin{tikzcd}
\bfD^b(\A_{t-1}) \arrow[tail]{rr}&&\bfD^b(\A_t)
  \arrow[twoheadrightarrow,yshift=-1.5ex]{ll}
  \arrow[twoheadrightarrow,yshift=1.5ex]{ll}
  \arrow[twoheadrightarrow]{rr} &&\bfD^\perf(k)\, .
  \arrow[tail,yshift=-1.5ex]{ll}
  \arrow[tail,yshift=1.5ex]{ll}
\end{tikzcd}
\]
\end{prop}
\begin{proof}
  The diagram \eqref{eq:rec-khwt} induces the recollement of derived
  categories.  For the right half of the diagram, this is clear since
  all functors are exact. The inclusion $\A_{t-1}\to\A_t$ induces a
  fully faithful functor $i_*\colon \bfD^b(\A_{t-1})\to\bfD^b(\A_t)$
  because the recollement in Theorem~\ref{th:khwt} is homological. We
  obtain the left adjoint of $i_*$ by completing the counit
  $j_!j^!(X)\to X$ to an exact triangle in $\bfD^b(\A_t)$, and
  analogously the right adjoint by completing the unit
  $X\to j_*j^*(X)$.

  Using these recollements, an induction shows that each inclusion
  $\proj\La_t\to\mod(\La_t,k)=\A_t$ induces a triangle equivalence
  $\bfD^\perf(\La_t)\xto{\sim} \bfD^b(\A_t)$. The functor is exact and
  fully faithful; so we show by induction that $\La_t$ generates
  $\bfD^b(\A_t)$ as a triangulated category.

  Each object $X\in\bfD^b(\A_t)$ fits into an exact triangle
  $j_!j^!(X)\to X\to i_*i^*(X)\to$, and we claim that $j_!j^!(X)$ and
  $i_*i^*(X)$ belong to $\bfD^\perf(\La_t)$. This is clear for
  $j_!j^!(X)$, because it is generated by $\De_t$ which is a finitely
  generated projective $\La_t$-module. On the other hand, we have an
  exact triangle $\De_t^r\to \La_t\to\La_{t-1}\to$ for some $r\ge 0$
  by Lemma~\ref{le:standard-coloc}, and therefore $\La_{t-1}$ is in
  $\bfD^\perf(\La_t)$.  Thus $i_*i^*(X)$ belongs to
  $\bfD^\perf(\La_{t-1})\subseteq \bfD^\perf(\La_t)$ by the
  induction hypothesis.
\end{proof}

Let $\Filt^\oplus(\De_1,\ldots,\De_n)$ denote the idempotent
completion of $\Filt(\De_1,\ldots,\De_n)$. 

\begin{cor}\label{co:kgldim}
The sequence of inclusion functors
\[\proj\La\to \Filt^\oplus(\De_1,\ldots,\De_n)\to \A\]
induces  triangle equivalences
\[\bfD^\perf(\La) \xto{\sim}\bfD^b(\Filt^\oplus(\De_1,\ldots,\De_n))
\xto{\sim}\bfD^b(\A).\]
Analogously, the inclusion
$\Filt^\oplus(\nabla_1,\ldots,\nabla_n)\to \A$
induces  a triangle equivalence
\[\bfD^b(\Filt^\oplus(\nabla_1,\ldots,\nabla_n))
\xto{\sim}\bfD^b(\A).\]
\end{cor}
\begin{proof}
  The argument for the inclusion
  $\proj\La\to\Filt^\oplus(\De_1,\ldots,\De_n)$ is precisely that
  given for $\proj\La\to\A$ in Proposition~\ref{pr:kgldim}, using the
  derived version of the colocalisation sequence
  \eqref{eq:standard-coloc}. The assertion for
  $\Filt^\oplus(\nabla_1\ldots,\nabla_n)$ follows from the first part
  by duality since $\Hom_k(-,k)$ maps $\De_i$ to $\nabla_i$.
\end{proof}

\begin{rem}\label{re:kgldim}
  The triangle equivalence $\bfD^\perf(\La)\xto{\sim}\bfD^b(\A)$
  implies that every object in $\A$ has finite projective and finite
  injective dimension. 
\end{rem}

\section{Ringel duality}\label{se:ringel}

There is a special class of tilting modules for any quasi-hereditary
artin algebra which Ringel introduced in \cite{Ri1991}. This was later
extended to highest weight categories over more general base rings
\cite{Do1993,Ro2008}.

Let $k$ be a commutative ring. We fix a $k$-linear highest weight
category $\A$ with standard objects $\De_1,\ldots,\De_n$ and
costandard objects $\nabla_1,\ldots,\nabla_n$. To simplify notation we
set
\[\Filt(\De)=\Filt(\De_1,\ldots,\De_n)\qquad\text{and}\qquad
\Filt(\nabla)=\Filt(\nabla_1,\ldots,\nabla_n).\]

Given any set
$X_1,\ldots,X_t$ of objects in $\A$, we write
$\Filt^\oplus(X_1,\ldots,X_t)$ for the closure of
$\Filt(X_1,\ldots,X_t)$ under direct summands.

\subsection*{Ext-orthogonality}

We compute the extensions groups between standard and costandard
objects. The first lemma is an immediate consequence of the
definition of a highest weight category.

\begin{lem}\label{le:Ext-orth}
For $1\le s,t\le n$ and $p\ge 0$ we have
\[\Ext^p_\A(\De_s,\nabla_t)\cong
\begin{cases}
k\qquad \text{if }s=t\text{ and }p=0,\\
0 \qquad \text{otherwise.}
\end{cases}
\]
\end{lem}
\begin{proof}
  We use induction on $n$. For $s,t<n$ the assertion follows by
  induction, because $\De_s,\nabla_t\in\A_{n-1}$ and the inclusion
  $\A_{n-1}\to\A_n=\A$ induces a homological recollement; see
  Theorem~\ref{th:khwt}. If $s=n$ or $t=n$, then we use the fact that
  $\De_n$ is projective and $\nabla_n$ is injective. This gives the
  assertion for $p>0$. For $p=0$ we use the recollement
  \eqref{eq:khwt-rec}. In fact, $\De_n=j_!(k)$ and
  $\nabla_n=j_*(k)$. Thus $\Hom_\A(\De_n,\nabla_n)\cong k$ by
  adjointness.
\end{proof}

\begin{cor}\label{co:Ext-orth}
  For $X\in\Filt^\oplus(\De)$ and
  $Y\in\Filt^\oplus(\nabla)$ we have
  $\Ext^p_\A(X,Y)=0$ for all $p>0$ and the $k$-module $\Hom_\A(X,Y)$
  is finitely generated projective.\qed
\end{cor}

\begin{prop}\label{pr:Ext-orth}
  Let $\A$ be a highest weight category. For $X$ in $\A$ we have:
\begin{enumerate}
\item $X\in\Filt^\oplus(\De)$ if and only if $\Ext_\A^1(X,\nabla_t)=0$
  for $1\le t\le n$. 
\item $X\in\Filt^\oplus(\nabla)$ if and only if $\Ext_\A^1(\De_t,X)=0$
  for $1\le t\le n$.
\end{enumerate}
\end{prop}
\begin{proof}
  We prove (1) and the proof of (2) is dual. One direction is clear by
  Corollary~\ref{co:Ext-orth}. Thus assume that
  $\Ext_\A^1(X,\nabla_t)=0$ for all $t$. We use induction on $n$ and
  consider the recollement \eqref{eq:khwt-rec}. First observe that the
  counit $j_!j^!(X)\to X$ is a monomorphism. To see this, fix an
  injective cogenerator $Q$ of $\A$. Note that $Q$ belongs to
  $\Filt^\oplus(\nabla_1,\ldots,\nabla_n)$. Thus we have an exact
  sequence \[0\lto i_*i^*(Q)\lto Q\lto j_!j^!(Q)\lto 0\]
  which induces the following commutative diagram with exact rows.
\[
\begin{tikzcd}[column sep=tiny]
0\arrow{r}&\Hom_\A(X,i_!i^!(Q))\arrow{d}\arrow{r}&\Hom_\A(X,Q)\arrow{d}\arrow{r}
&\Hom_\A(X,j_*j^*(Q))\arrow{d}\arrow{r}&0\\
0\arrow{r}&\Hom_\A(j_!j^!(X),i_!i^!(Q))\arrow{r}&\Hom_\A(j_!j^!(X),Q)\arrow{r}
&\Hom_\A(j_!j^!(X),j_*j^*(Q))\arrow{r}&0
\end{tikzcd}
\]
We have $\Hom_\A(j_!j^!(X),i_!i^!(Q))=0$ and the
map \[\Hom_\A(X,j_*j^*(Q))\lto \Hom_\A(j_!j^!(X),j_*j^*(Q))\]
is a bijection by adjointness. Thus the map
\[\Hom_\A(X,Q)\lto \Hom_\A(j_!j^!(X),Q)\]
is surjective. It follows that the sequence
\[0\lto j_!j^!(X)\lto X\lto i_*i^*(X)\lto 0\]
given by the unit and counit for $X$ is exact. The object
$X'=i_*i^*(X)$ belongs to $\A_{n-1}$ and satisfies again
$\Ext_\A^1(X',\nabla_t)=0$ for all $t$. Thus $X'$ belongs to
$\Filt^\oplus(\De_1,\ldots,\De_{n-1})$ by induction. It follows that
$X$ belongs to $\Filt^\oplus(\De_1,\ldots,\De_n)$.
\end{proof}

\begin{rem}
A consequence of Proposition~\ref{pr:Ext-orth} is the fact that
the subcategory $\Filt^\oplus(\De)$ of $\A$ is closed under taking
kernels of epimorphisms.
\end{rem}

\subsection*{Tilting objects}

We describe the special tilting objects for a  $k$-linear highest weight category.

\begin{prop}\label{pr:tilting}
  Let $\A$ be a $k$-linear highest weight category with costandard
  objects $\nabla_1,\ldots,\nabla_n$. Then there are finitely many exact
  sequences
\begin{equation*}
0\lto V_i\lto T_i\lto \nabla_i \lto 0\qquad (1\le i\le n)
\end{equation*}
  in $\A$ satisfying the following:
\begin{enumerate}
\item $V_i$ belongs to $\Filt(\nabla_1,\ldots,\nabla_{i-1})$ for all $i$.
\item $T=\bigoplus_{i=1}^n T_i$ is a projective generator of
  $\Filt(\nabla_1,\ldots,\nabla_{n})$.
\item $\End_\A(T)$ is finitely generated projective over $k$.
\end{enumerate}
\end{prop}
\begin{proof}
  The costandard objects satisfy $\Ext_\A^1(\nabla_j,\nabla_i)=0$ for
  all $i\ge j$ because of the duality \eqref{eq:duality}. Now apply
  Lemma~\ref{le:standard}. The object $T$ belongs to
  $\Filt^\oplus(\De_1,\ldots,\De_n)$ by
  Proposition~\ref{pr:Ext-orth}. Thus $\End_\A(T)$ is finitely
  generated projective over $k$ by Corollary~\ref{co:Ext-orth}.
\end{proof}

We formulate some immediate consequences of
Proposition~\ref{pr:tilting}.

For an object $X$ in an additive category we denote by $\add X$ the
full subcategory whose objects are the direct summands of finite
direct sums of copies of $X$.

\begin{cor}\label{co:tilting-characterisation}
  Let $\A$ be a $k$-linear highest weight category. For an object $T$
  in $\A$ the following are equivalent:
\begin{enumerate}
\item $T$ is a projective generator of 
  $\Filt^\oplus(\nabla)$. 
\item $T$ is an injective cogenerator of 
  $\Filt^\oplus(\De)$. 
\item  $\Filt^\oplus(\De)\cap
\Filt^\oplus(\nabla)=\add T$.
\end{enumerate}
 \end{cor}
\begin{proof}
 Combine  Propositions~\ref{pr:Ext-orth} and \ref{pr:tilting}.
\end{proof}

\begin{cor}\label{co:tilting}
Let $\A$ be a $k$-linear highest weight category $\A$ with costandard
  objects $\nabla_1,\ldots,\nabla_n$ and fix a projective generator
  $T$ of $\Filt(\nabla_1,\ldots,\nabla_{n})$. Set $\La'=\End_\A(T)$
  and $\De'_i=\Hom_\A(T,\nabla_{n-i})$.  Then $\mod(\La',k)$ is a
  $k$-linear highest weight category with standard objects
  $\De'_1,\ldots,\De'_n$ and $\Hom_\A(T,-)$ induces an equivalence
\begin{equation}\label{eq:tilting}
 \Filt(\nabla_1,\ldots,\nabla_{n})\xto{\sim}\Filt(\De'_1,\ldots,\De'_{n})
\end{equation}
of exact categories.\qed
\end{cor}

The highest weight category $\mod(\La',k)$ in
Corollary~\ref{co:tilting} is called the \emph{Ringel dual} of $\A$.
If $\A\xto{\sim}\mod(\La,k)$ for some $k$-split quasi-hereditary
algebra $\La$, then the quasi-hereditary algebra $\La'$ is called the
\emph{Ringel dual} of $\La$; it is unique only up to Morita
equivalence.

\begin{prop}
  Let $\La$ be a $k$-split quasi-hereditary algebra. The double Ringel
  dual $\La''=(\La')'$ is Morita equivalent to $\La$. The equivalence
  identifies the standard modules over $\La''$ and $\La$.
\end{prop}
\begin{proof}
We have equivalences
\[\Filt(\De'')\xto{\sim}\Filt(\nabla')\xto{\sim}\Filt(\De')^\op
\xto{\sim}\Filt(\nabla)^\op\xto{\sim}\Filt(\De)\]
of exact categories. Restricting this equivalence to the full
subcategories of projective objects yields an equivalence
$\proj\La''\xto{\sim}\proj\La$.
\end{proof}

Recall that an object $T$ of an exact category $\A$ is a \emph{tilting
  object} if $\Ext_\A^p(T,T)=0$ for $p>0$ and $\bfD^b(\A)$ admits no
proper thick subcategory containing $T$. An equivalent statement is
that $\RHom_\A(T,-)$ induces a triangle equivalence
$\bfD^b(\A)\xto{\sim}\bfD^\perf(\End_\A(T))$. In that case a
quasi-inverse is denoted by $-\Lotimes_{\End_\A(T)} T$.

\begin{cor}\label{co:tilt}
  Let $\A$ be a $k$-linear highest weight category $\A$. Then a
  projective generator of $\Filt(\nabla)$ is a
  tilting object of $\A$.
\end{cor}

\begin{proof}
  Fix a projective generator $T$ and set $\La'=\End_\A(T)$.  Then the
  sequence of fully faithful exact functors
\[\proj\La'\xto{\sim}\add T\to \Filt^\oplus(\nabla)\to \A\]
induces a triangle equivalence
\[\bfD^\perf(\La')\xto{\sim}\bfD^b(\Filt^\oplus(\nabla))
\xto{\sim}\bfD^b(\A)\]
which is a quasi-inverse of $\RHom_\A(T,-)$; this follows
from Corollary~\ref{co:kgldim}.
\end{proof}

For a $k$-linear highest weight category $\A$ an object $T$ satisfying
the equivalent conditions in
Corollary~\ref{co:tilting-characterisation} is called
\emph{characteristic tilting object}.

\subsection*{Tor-orthogonality}

Ext-orthogonality for modules over a quasi-hereditary algebra
translates into Tor-orthogonality. To see this we need to recall some
standard isomorphisms for derived functors.

\begin{lem}\label{le:derived}
Let $\La$ be a $k$-algebra and $X,Y$ be complexes of
$\La$-modules. Then there are natural morphisms
\begin{gather*}
X\Lotimes_\La\RHom_k(Y,k)\lto
\RHom_k(\RHom_\La(X,Y),k)\\
Y\Lotimes_\La\RHom(X,\La)\lto \RHom_\La(X,Y)
\end{gather*}
which are isomorphisms when $X$ is perfect.\qed
\end{lem}

\begin{prop}\label{pr:tor}
Let $\La$ be a $k$-split quasi-hereditary algebra. 
For \[X\in\Filt^\oplus(\De)\subseteq\mod(\La,k)\qquad
\text{and}\qquad Y\in\Filt^\oplus(\De)\subseteq\mod(\La^\op,k)\]
we have
\[\Tor_p^\La(X,Y)=0\quad\text{for}\quad p>0.\] 
\end{prop}
\begin{proof}
  This follows from Corollary~\ref{co:Ext-orth} with the first
  isomorphism in Lemma~\ref{le:derived}, since $\Hom_k(-k)$ induces an
  equivalence $\Filt^\oplus(\nabla)\xto{\sim}\Filt^\oplus(\De)$; see
  \eqref{eq:duality}.
\end{proof}

\subsection*{Serre duality}

Let $\La$ be a $k$-algebra that is finitely generated projective over
$k$. Then the $\La$-module $\Hom_k(\La,k)$ is an injective cogenerator
of $\mod(\La,k)$ and plays the role of a dualising complex.

\begin{lem}\label{le:serre}
Suppose that the $\La$-module $\Hom_k(\La,k)$  has finite
projective dimension. Then
\[F=-\Lotimes_\La\Hom_k(\La,k)\colon\bfD^\perf(\La)\lto\bfD^\perf(\La)\]
is a \emph{Serre functor} in the sense that $F$ is a triangle
equivalence and 
\[\RHom_k(\RHom_\La(X,-),k)\cong\RHom_\La(-,F(X))\quad\text{for}\quad
X\in\bfD^\perf(\La).\]
\end{lem}
\begin{proof}
  Using the standard isomorphisms from Lemma~\ref{le:derived} we have
\begin{align*}
\RHom_k(\RHom_\La(X,-),k) &\cong \RHom_k(-\Lotimes_\La\RHom_\La(X,\La),k)\\
&\cong \RHom_\La(-,\RHom_k(\RHom_\La(X,\La),k))\\
&\cong \RHom_\La(-,X\Lotimes_\La\Hom_k(\La,k))
\end{align*}
and a quasi-inverse of $F$ is given by $\RHom_\La(\Hom_k(\La,k),-)$.
\end{proof}

\begin{prop}\label{pr:serre}
Let $\La$ be a $k$-split quasi-hereditary algebra. Then 
\[-\Lotimes_\La\Hom_k(\La,k)\colon\bfD^\perf(\La)\lto\bfD^\perf(\La)\]
is a Serre functor.
\end{prop}
\begin{proof}
  Combine Lemma~\ref{le:serre} with the fact that $\Hom_k(\La,k)$ has
  finite projective dimension; see Remark~\ref{re:kgldim}.
\end{proof}

Serre duality and Ringel duality are closely related for a
quasi-hereditary algebra. The following proposition provides the first
step for explaining this.

\begin{prop}\label{pr:ringel-dual}
  Let $\La$ be a $k$-split quasi-hereditary algebra with
  characteristic tilting module $T$ and set $\Ga=\End_\La(T)$. Then
  $\Hom_k(T,k)$ is a characteristic tilting module over both $\Ga$ and
  $\La^\op$, with canonical isomorphisms
\[\End_{\Ga}(\Hom_k(T,k))\cong\La\qquad\text{and}\qquad \End_{\La^\op}(\Hom_k(T,k))\cong\Ga^\op.\] 
Moreover, $T$ is a  characteristic tilting module over $\Ga^\op$ with $\End_{\Ga^\op}(T)\cong\La^\op$.
\end{prop}
\begin{proof}
For an exact category $\A$ we write $\proj\A$ and $\inj\A$ to denote
  the full subcategories of projective and injective objects,
  respectively.

The equivalence \eqref{eq:tilting} given by
  $\Hom_\La(T,-)$ restricts to an equivalence
\[\inj(\Filt^\oplus(\nabla))\xto{\sim}\inj(\Filt^\oplus(\De))=\add
T\]
and sends $\Hom_k(\La,k)$ to 
\[\Hom_\La(T,\Hom_k(\La,k))\cong \Hom_k(T,k).\]
Thus $\Hom_k(T,k)$ is a characteristic tilting module over $\Ga$ with
\[\End_\Ga(\Hom_k(T,k))\cong \End_{\La}(\Hom_k(\La,k))\cong\La.\]
On the other hand, the equivalence \eqref{eq:duality} given by $\Hom_k(-,k)$
restricts to an equivalence
\[(\add T)^\op=\inj(\Filt^\oplus(\De))^\op\xto{\sim}\proj(\Filt^\oplus(\nabla)).\]
Thus $\Hom_k(T,k)$ is a characteristic tilting module over $\La^\op$ with
\[\End_{\La^\op}(\Hom_k(T,k))\cong \End_{\La}(T)^\op\cong\Ga^\op.\]
The assertion for the $\Ga^\op$-module $T$ now follows since
$T\cong\Hom_k(\Hom_k(T,k),k)$.
\end{proof}

\begin{thm}\label{th:ringel-dual}
  Let $\La$ be a $k$-split quasi-hereditary algebra with
 characteristic tilting module $T$ and set $\Ga=\End_\La(T)$. Then
 \[\Hom_k(T,k) \otimes_\Ga T\cong \Hom_k(T,k) \Lotimes_\Ga T\cong\Hom_k(\La,k)\]
 as $\La$-$\La$-bimodules. Therefore the composite
\[
\begin{tikzcd}
\bfD^\perf(\La)\arrow{rrr}{-\Lotimes_\La \Hom_k(T,k)}&&&\bfD^\perf(\Ga)
\arrow{rrr}{-\Lotimes_\Ga T}&&&\bfD^\perf(\La)
\end{tikzcd}
\]
is a Serre functor.
\end{thm}
\begin{proof}
  We apply Proposition~\ref{pr:ringel-dual}. The modules $T$ and
  $\Hom_k(T,k)$ over $\Ga$ are characteristic tilting modules; this
  yields the first isomorphism by Proposition~\ref{pr:tor}. The second
  isomorphism follows from Lemma~\ref{le:derived}. The description of
  the Serre functor then follows by Lemma~\ref{le:serre}.
\end{proof}

\subsection*{Ringel self-dual algebras} 
The connection between Serre duality and Ringel duality is of
particular interest for a quasi-hereditary algebra that is Ringel
self-dual.

\begin{defn}
We say that a quasi-hereditary algebra $\La$ is \emph{Ringel
  self-dual} if it satisfies one of the following equivalent
conditions:
\begin{enumerate}
\item The highest weight category $\mod(\La,k)$ is equivalent to its
  Ringel dual.
\item There is a characteristic tilting module $T$ over $\La$ and an
  isomorphism $\La'=\End_\La(T)\xto{\sim}\La$ which identifies the
  standard modules over $\La'$ and $\La$.
\end{enumerate}
\end{defn}

The following description of Ringel duality as a square root of Serre
duality is inspired by a result for strict polynomial functors
\cite{Kr2013} and a similar result in the context of the
Bernstein-Gelfand-Gelfand category $\mathcal O$ \cite{MS2008}. 

Let us fix for a Ringel self-dual algebra $\La$ a characteristic
tilting module $T$ as in the above definition and identify
$\End_\La(T)=\La$. This turns $T$ and $\Hom_k(T,k)$ into
$\La$-$\La$-bimodules.

\begin{cor}
  Let $\La$ be a $k$-split quasi-hereditary algebra. Suppose that
  $\La$ is Ringel self-dual with characteristic tilting module
  $T$. Then the following are equivalent:
\begin{enumerate}
\item  $T\cong\Hom_k(T,k)$ as $\La$-$\La$-bimodules.
\item $T\Lotimes_\La T\cong \Hom_k(\La,k)$ as $\La$-$\La$-bimodules.
\item $(-\Lotimes_\La T)^2$ is a Serre functor for $\bfD^\perf(\La)$.
\end{enumerate}
\end{cor}
\begin{proof}
Apply Theorem~\ref{th:ringel-dual}.
\end{proof}

\part{Strict polynomial functors}

In the second part of these notes we explain the highest weight
structure for categories of strict polynomial functors \cite{FS1997},
working over an arbitrary commutative ring and using some of the
principal results from the theory of Schur and Weyl functors
\cite{ABW1982}.

\section{Divided powers and strict polynomial functors}

Strict polynomial functors were introduced by Friedlander and Suslin
\cite{FS1997}. In this section we recall the definition and some basic
properties, using an equivalent description in terms of
representations of divided powers. For details and further references,
see \cite{Kr2013, Ku1998, Pi2003,Ka2012}.  The material is elementary,
based to a large extent on classical facts from multilinear
algebra. In particular, properties of divided powers are used, for
which we refer to \cite[IV.5]{Bo1981}. The language of strict
polynomial functors is employed because of its flexibility. Evaluating
strict polynomial functors at a free module of finite rank makes it
easy to transfer this work to the representation theory of Schur
algebras.

\subsection*{Finitely generated projective modules}
Throughout we fix a commutative ring $k$. Let $\P_k$ denote the
category of finitely generated projective $k$-modules. Given $V,W$ in
$\P_k$, we write $V\otimes W$ for their tensor product over $k$ and
$\Hom(V,W)$ for the group of $k$-linear maps $V\to W$. This provides
two bifunctors
\begin{align*}
-\otimes -&\colon\P_k\times\P_k\lto\P_k\\
\Hom(-,-)&\colon(\P_k)^\op\times\P_k\lto\P_k
\end{align*}
and the functor sending $V$ to $V^*=\Hom(V,k)$ yields a duality 
\[(\P_k)^\op \stackrel{\sim}\lto\P_k.\]

\subsection*{Divided and symmetric powers}

Fix a positive integer $d$ and denote by $\frS_d$ the symmetric group
permuting $d$ elements. For each $V\in\P_k$, the group $\frS_d$ acts
on $V^{\otimes d}$ by permuting the factors of the tensor product.
Denote by $\Ga^dV$ the submodule $(V^{\otimes d})^{\frS_d}$ of
$V^{\otimes d}$ consisting of the elements which are invariant under
the action of $\frS_d$; it is called the module of \emph{divided
  powers} (more correctly: \emph{symmetric tensors}) of degree $d$.  The
maximal quotient of $V^{\otimes d}$ on which $\frS_d$ acts trivially
is denoted by $S^dV$ and is called the module of \emph{symmetric
  powers} of degree $d$.  Set $\Ga^0V=k$ and $S^0V=k$.

From the definition, it follows that $(\Ga^dV)^*\cong S^d (V^*)$.
Note that $S^d V$ is a free $k$-module provided that $V$ is free. Thus
$\Ga^d V$ and $S^d V$ belong to $\P_k$ for all $V\in\P_k$, and we
obtain functors $\Ga^d,S^d\colon\P_k\to\P_k$.

\subsection*{The category of divided powers}
We consider the category $\Ga^d\P_k$ which is defined as follows. The
objects are the finitely generated projective $k$-modules and for two
objects $V,W$ set
 \[\Hom_{\Ga^d\P_k}(V,W)=\Ga^d\Hom(V,W).\]
 This identifies with $\Hom(V^{\otimes d},W^{\otimes d})^{\frS_d}$,
 where $\frS_d$ acts on $\Hom(V^{\otimes d},W^{\otimes d})$ via $(\s
 f)(v)=\s^{-1} f(\s v)$ for $f\colon V^{\otimes d}\to W^{\otimes d}$
 and $\s\in\frS_d$.  Using this identification one defines the
 composition of morphisms in $\Ga^d\P_k$. The duality for $\P_k$
 induces a duality \[(\Ga^d\P_k)^\op\stackrel{\sim}\lto \Ga^d\P_k.\]

\begin{exm}\label{ex:green}
  Let $n$ be a positive integer and set $V=k^n$. Then
  $\End_{\Ga^d\P_k}(V)$ is isomorphic to the \emph{Schur algebra}
  $S_k(n,d)$ as defined by Green \cite[Theorem~2.6c]{Gr1980}. 
\end{exm}

Following \cite{Ka2012}, this example suggests for $\Ga^d\P_k$ the
term \emph{Schur category}.

\subsection*{Strict polynomial functors}

Let $\M_k$ denote the category of $k$-modules. We study the category
of $k$-linear \emph{representations} of $\Ga^d\P_k$. This is by
definition the category of $k$-linear functors $\Ga^d\P_k\to\M_k$ and
we write by slight abuse of notation
\[\Rep\Ga^d_k=\Fun_k(\Ga^d\P_k,\M_k).\]
For  objects $X,Y$ in $\Rep\Ga^d_k$ the set of morphisms is
denoted by $\Hom_{\Ga^d_k}(X,Y)$.

The representations of $\Ga^d\P_k$ form an abelian category, where
(co)kernels and (co)products are computed pointwise in the category of
$k$-modules.

\subsection*{The Yoneda embedding}
The Yoneda embedding
\[(\Ga^d\P_k)^\op\lto\Rep\Ga^d_k,\quad V\mapsto\Hom_{\Ga^d\P_k}(V,-)\]
identifies $\Ga^d\P_k$ with the full subcategory
consisting of the representable functors. For  $V\in\Ga^d\P_k$ we write
\[\Ga^{d,V}=\Hom_{\Ga^d\P_k}(V,-).\]
For $X\in\Rep\Ga^d_k$ there is the Yoneda isomorphism
\[\Hom_{\Ga^d_k}(\Ga^{d,V},X)\xto{\sim} X(V)\] 
and it follows that $\Ga^{d,V}$ is a projective
object in $\Rep\Ga^d_k$.

\subsection*{Duality}
Given a  representation $X\in\Rep\Ga^d_k$, its \emph{dual}  $X^\circ$ is defined by
\[X^\circ(V)=X(V^*)^*.\] 
We have  for all $X,Y\in\Rep\Ga^d_k$ a natural isomorphism
\[\Hom_{\Ga^d_k}(X,Y^\circ)\cong\Hom_{\Ga^d_k}(Y,X^\circ).\]
The evaluation morphism $X\to X^{\circ\circ}$ is an
isomorphism when $X$ takes values in $\P_k$.

\begin{exm} The divided power functor $\Ga^d$ and the symmetric power
  functor $S^d$ belong to $\Rep\Ga^d_k$. In fact
\[\Ga^d=\Hom_{\Ga^d\P_k}(k,-)\quad\text{and}\quad S^d\cong(\Ga^d)^\circ.\]
\end{exm}

\subsection*{The algebra of divided powers} 

Given $V\in\P_k$, we set $\Ga V=\bigoplus_{d\ge 0}\Ga^dV$.  For
non-negative integers $d,e$ the inclusion
$\frS_{d}\times\frS_e\subseteq \frS_{d+e}$ induces natural maps
\begin{equation}\label{eq:comult}
\Ga^{d+e}V\lto \Ga^{d}V\otimes \Ga^{e}V\qquad\text{and}\qquad
\Ga^{d}V\otimes \Ga^{e}V\lto \Ga^{d+e}V.
\end{equation}
The first map is given by \[(V^{\otimes d+e})^{\frS_{d+e}}\subseteq
(V^{\otimes d+e})^{\frS_d\times\frS_e} \cong (V^{\otimes
  d})^{\frS_d}\otimes (V^{\otimes e})^{\frS_e}.
\]
The second map sends $x\otimes y\in\Ga^{d}V\otimes \Ga^{e}V$
to \[xy=\sum_{g\in \frS_{d+e}/\frS_{d}\times\frS_e} g(x\otimes y)\]
where $g(x\otimes y)=\s(x\otimes y)$ for a coset
$g=\s(\frS_{d}\times\frS_e)$. This multiplication gives $\Ga V$ the
structure of a commutative $k$-algebra.

Now suppose that $V$ is a free $k$-module with basis
$\{v_1,\ldots,v_n\}$. Let $\La(n,d)$ denote the set of sequences
$\la=(\la_1,\la_2,\ldots,\la_n)$ of non-negative integers such that
$\sum\la_i=d$.  Then the elements \[v_\la=\prod_{i=1}^n
v_i^{\otimes\la_i}\quad\text{for}\quad\la\in\La(n,d)\] form a
$k$-basis of $\Ga^d V$.

Let $\{v^*_1,\ldots,v^*_n\}$ denote the dual basis of $V^*$. We
identify the symmetric algebra $S(V^*)=\bigoplus_{d\ge
  0}S^d(V^*)$ with the polynomial algebra
$k[v^*_1,\ldots,v^*_n]$. Let $\{v^*_\la\}_{\la\in\La(n,d)}$ be the
basis of $(\Ga^d V)^*$ dual to $\{v_\la\}_{\la\in\La(n,d)}$. Then the
canonical isomorphism $(\Ga^d V)^*\xto{\sim}S^d(V^*)$ maps each
$v^*_\la$ to $\prod_{i=1}^n (v^*_i)^{\la_i}$.

\subsection*{Tensor products}
For non-negative integers $d,e$ there is a tensor product \[-\otimes
-\colon\Rep\Ga^d_k\times \Rep\Ga^e_k\lto \Rep\Ga^{d+e}_k.\]
Let $X\in\Rep\Ga^d_k$ and $Y\in\Rep\Ga^e_k$. The
functor $X\otimes Y$ acts on objects via
\[(X\otimes Y)(V)=X(V)\otimes Y(V)\] and on morphisms via the map
\[\Ga^{d+e}\Hom(V,W)\lto \Ga^{d}\Hom(V,W)\otimes \Ga^{e}\Hom(V,W)\]
given by \eqref{eq:comult}. Note that 
\[(X\otimes Y)^\circ\cong X^\circ\otimes Y^\circ\] when $X$ and $Y$
take values in $\P_k$.

For $\la\in\La(n,d)$ we
set \[\Ga^\la=\Ga^{\la_1}\otimes\cdots\otimes\Ga^{\la_n}\quad
\text{and}\quad S^\la=S^{\la_1}\otimes\cdots\otimes S^{\la_n}.\] We
have \[(\Ga^\la)^\circ\cong S^\la\qquad\text{and}\qquad
\Ga^{(1,\ldots,1)}\cong\otimes ^n\cong S^{(1,\ldots,1)}.\]

\subsection*{Graded representations} 

It is sometimes convenient to consider the category 
\[\prod_{d\ge 0}\Rep\Ga^d_k\]
consisting of graded representations $X=(X^0,X^1,X^2,\ldots)$. An
example is for each $V\in\P_k$ the representation
\[\Ga^V=(\Ga^{0,V},\Ga^{1,V},\Ga^{2,V},\ldots).\] The tensor
product $X\otimes Y$ of graded representations $X,Y$ is defined in
degree $d$ by
\[(X\otimes Y)^d=\bigoplus_{i+j=d}X^i\otimes Y^j.\]

\subsection*{Decomposing divided powers}

The assignment which takes $V\in\P_k$ to the symmetric algebra
$SV=\bigoplus_{d\ge 0}S^dV$ gives a functor from $\P_k$ to the
category of commutative $k$-algebras which preserves coproducts. Thus
\[S V\otimes S W\cong S(V\oplus W)\]
and therefore by duality
\[\Ga V\otimes \Ga W\cong\Ga (V\oplus W).\]
This yields an isomorphism of graded representations
\[\Ga^{V}\otimes\Ga^{W}\cong \Ga^{V\oplus W}.\]
Thus for each positive integer $n$, one obtains in degree $d$ a decomposition
\[ \Ga^{d,k^n}=\bigoplus_{i=0}^d(\Ga^{d-i,k^{n-1}}\otimes \Ga^i)\]
and using induction a canonical decomposition
\begin{equation}\label{eq:decomp}
  \Ga^{d,k^n}=\bigoplus_{\la\in\La(n,d)}\Ga^{\la}.
\end{equation}

The decomposition of divided powers implies that the finitely
generated projective objects in $\Rep\Ga^d_k$ are precisely the direct
summands of finite direct sums of functors $\Ga^\la$, where
$\la=(\la_1,\ldots,\la_n)$ is any sequence of non-negative integers
satisfying $\sum\la_i=d$ and $n$ is any positive integer.

\subsection*{Exterior powers}

Given $V\in\P_k$, let $\La V=\bigoplus_{d\ge 0}\La^d V$ denote the
exterior algebra, which is obtained from the tensor algebra
$TV=\bigoplus_{d\ge 0}V^{\otimes d}$ by taking the quotient with
respect to the ideal generated by the elements $v\otimes v$, $v\in V$.

For each $d\ge 0$, the $k$-module $\La^dV$ is free provided that $V$
is free. Thus $\La^d V$ belongs to $\P_k$ for all $V\in\P_k$, and this
gives a functor $\Ga^d\P_k\to\P_k$, since the ideal generated by the
elements $v\otimes v$ is invariant under the action of $\mathfrak S_d$
on $V^{\otimes d}$. There is a natural isomorphism \[\La^d (V^*)\cong
(\La^d V)^*\] induced by $(f_1\wedge\cdots\wedge
f_d)(v_1\wedge\cdots\wedge v_d)=\det (f_i(v_j))$, and therefore
$(\La^d)^\circ\cong\La^d$.

\subsection*{Representations of Schur algebras}

Strict polynomial functors and modules over Schur algebras are closely
related, since for any $X\in \Rep\Ga^d_k$ the Schur algebra $S_k(n,d)$
acts on $X(k^n)$; cf.\ Example~\ref{ex:green}.

Let $n\ge d$. The functor
\begin{equation}\label{eq:schur}
\Rep\Ga^d_k\lto \Mod S_k(n,d)^\op,\quad X\mapsto X(k^n)
\end{equation}
gives  an equivalence, because evaluation at $k^n$ identifies
with $\Hom_{\Ga^d_k}(\Ga^{d,k^n},-)$ and $\Ga^{d,k^n}$ is
a projective generator.\footnote{Our preference is to work
  with the functor category $\Rep\Ga^d_k$, because the Schur category
  $\Ga^d\P_k$ carries useful structure (e.g.\ $\oplus$ or $\otimes$)
  which `disappears' when one evaluates at a single object $k^n$.}

\subsection*{Base change}

Let $k\to k'$ be a homomorphism of commutative rings. The functor
$-\otimes_k k'\colon \P_k\to \P_{k'}$ induces for each positive
integer $d$  functors \[\Ga^d\P_k\lto\Ga^d\P_{k'}\qquad\text{and}
\qquad\Rep\Ga^d_k\lto\Rep\Ga^d_{k'}\] which we denote again by
$-\otimes_k k'$. For example, $\Ga^\la\otimes_k k'=\Ga^\la$ for each
$\la\in\La(n,d)$.

We note that most results in this work are invariant under base change.

\section{Schur and Weyl functors}

Generalising the results of Schur \cite{Sc1901} and Lascoux
\cite{La1977} in characteristic zero, Schur and Weyl functors, in
arbitrary characteristic, were introduced by Akin, Buchsbaum, and
Weyman \cite{ABW1982}. We give the definition and refer to the next
section for a description in terms of (co)standard objects.

\subsection*{Partitions and Young diagrams}

Fix a positive integer $d$. A \emph{partition} of \emph{weight} $d$
(or simply a partition of $d$) is a sequence
$\la=(\la_1,\la_2,\ldots )$ of non-negative integers satisfying
$\la_1\ge\la_2\ge \ldots$ and $\sum\la_i=d$. Its \emph{conjugate}
$\la'$ is the partition where $\la'_i$ equals the number of terms of
$\la$ that are greater or equal than $i$.

Fix a partition $\la$ of weight $d$. Each integer $r\in\{1,\ldots,d\}$
can be written uniquely as sum $r=\la_1+\ldots\la_{i-1}+j$ with $1\le j\le
\la_i$. The pair $(i,j)$ describes the position ($i$th row and $j$th
column) of $r$ in the \emph{Young diagram} corresponding to $\la$. The
partition $\la$ determines a permutation $\s_\la\in\frS_d$ by
$\s_\la(r)=\la'_1+\ldots\la'_{j-1}+i$, where $1\le i\le\la_j$. Note
that $\s_{\la'}=\s_\la^{-1}$. Here is an example.
\[\la=(3,2)\quad
\ytableausetup{smalltableaux} \begin{ytableau}
1 & 2 & 3 \\ 4 & 5
\end{ytableau}
\qquad\la'=(2,2,1)\quad
\ytableausetup{smalltableaux} \begin{ytableau}
1 & 2 \\ 3 & 4\\ 5
\end{ytableau} \qquad \s_\la=\left(\begin{smallmatrix}1&2&3&4&5\\[.4em]
    1&3&5&2&4\end{smallmatrix}\right)
\]

\subsection*{Schur and Weyl modules} 

Fix a partition $\la$ of weight $d$, and assume that $\la_1+\cdots
+\la_n=d=\la'_1+\cdots +\la'_m$.  For $V\in\P_k$ one defines the
\emph{Schur module} $S_\la V$ as image of the map
\[\La^{\la'_1}V\otimes\cdots\otimes
\La^{\la'_m}V\xto{\Delta\otimes\cdots\otimes\Delta} V^{\otimes
  d}\xto{s_{\la}}V^{\otimes d}\xto{\nabla\otimes\cdots\otimes\nabla}
S^{\la_1}V\otimes\cdots\otimes S^{\la_n}V.\] Here, we denote for an
integer $r$ by $\Delta\colon\La^{r}V\to V^{\otimes r}$ the
comultiplication given by
\[\Delta(v_1\wedge\cdots\wedge v_{r})=\sum_{\s\in\frS_{r}}
{\sgn(\s)}v_{\s(1)}\otimes\cdots\otimes v_{\s(r)},\]
$\nabla\colon V^{\otimes r}\to S^r V$ is the multiplication, and
$s_\la\colon V^{\otimes d}\to V^{\otimes d}$ is given by
\[s_\la(v_1\otimes\cdots\otimes v_d)=v_{\s_\la(1)}\otimes\cdots\otimes
v_{\s_\la(d)}.\]
The corresponding \emph{Weyl module} $W_\la V$ is by definition the
image of the analogous map
\[\Ga^{\la_1}V\otimes\cdots\otimes
\Ga^{\la_n}V\xto{\Delta\otimes\cdots\otimes\Delta} V^{\otimes
  d}\xto{s_{\la'}}V^{\otimes d}\xto{\nabla\otimes\cdots\otimes\nabla}
\La^{\la'_1}V\otimes\cdots\otimes \La^{\la'_m}V.\] Note that
$(W_\la V)^*\cong S_{\la}(V^*)$.

\subsection*{Young tableaux}
Suppose that $V$ is a free $k$-module with basis
$\{v_1,\ldots,v_r\}$. We fix a partition $\la=(\la_1,\ldots,\la_n)$
and describe an explicit basis for $S_\la V$ and $W_\la V$.

A \emph{filling} of a Young diagram is a map which assigns to each box
a positive integer. A \emph{Young tableau} is a filling that is weakly
increasing along each row and strictly increasing down each column.

Each filling $T$ with entries in $\{1,\ldots,r\}$ yields two
elements \[v_T\in \Ga^{\la_1}V\otimes\cdots \otimes
\Ga^{\la_n}V\quad\text{and}\quad \hat v_T\in \La^{\la'_1}V\otimes\cdots
\otimes \La^{\la'_m}V\] by replacing any $i$ in a box by $v_i$. Here is
an example of a Young tableau
\begin{equation*}\label{eq:tableau}
\la=(5,3,3,2)\qquad
\ytableausetup{smalltableaux} \begin{ytableau}
1 & 2 & 2 &3&3\\
2&3&5\\
4&4&6\\
5&6
\end{ytableau}
\end{equation*}
and here are the corresponding elements.
\begin{align*}
v_T&=(v_1 ( v_2 \otimes v_2)(  v_3\otimes  v_3))\otimes
(v_2  v_3  v_5)\otimes ((v_4\otimes  v_4)  v_6)\otimes
(v_5  v_6)\\
\hat v_T&=(v_1\wedge v_2\wedge v_4\wedge v_5)\otimes
(v_2\wedge v_3\wedge v_4\wedge v_6)\otimes (v_2\wedge v_5\wedge v_6)\otimes
v_3\otimes v_3
\end{align*}
More precisely, let $T(i,j)$ denote the entry of the box $(i,j)$ and
define $\a^i\in\La(r,\la_i)$ by setting $\a^i_j=\card\{t\mid
T(i,t)=j\}$.  Then $v_T=v_{\a^1}\otimes\cdots\otimes
v_{\a^n}$. Note that the elements $v_T$ form a $k$-basis of
$\Ga^\la V$ as $T$ runs through all fillings (weakly increasing along
each row).

\begin{prop}[{\cite[Theorems~II.2.16 and II.3.16]{ABW1982}}]\label{pr:basis}
  \pushQED{\qed} Let $\la$ be a partition and $V$ a free $k$-module of
  rank $r$.
\begin{enumerate}
\item The canonical map $\La^{\la'} V\to S_\la V$ sends the elements
  $\hat v_T$ with $T$ a Young tableau on $\la$ with entries in
  $\{1,\ldots,r\}$ to a $k$-basis of $S_\la V$.
\item The canonical map $\Ga^\la V\to W_\la V$ sends the elements
  $v_T$ with $T$ a Young tableau on $\la$ with entries in
  $\{1,\ldots,r\}$ to a $k$-basis of $W_\la V$.\qedhere
\end{enumerate}
\end{prop}
For expositions on Schur and Weyl modules, see \cite[\S8.1]{Fu1997} or
\cite[\S2.1]{We2003}. There one finds proofs of
Proposition~\ref{pr:basis} and presentations of these modules, which
are relevant for the proof of Theorem~\ref{th:weight}.

\subsection*{Schur and Weyl functors}
The definition of Schur and Weyl modules gives rise to functors $S_\la$
and $W_\la$ in $\Rep\Ga^d_k$ for each partition $\la$ of weight
$d$. Note that $S_\la^\circ\cong W_\la$ and  $W_\la^\circ\cong S_\la$.

\begin{exm}
We have $S_{(1,\ldots,1)}=\La^d$ and $S_{(d)}=S^d$.
\end{exm}

\section{Weight spaces and (co)standard objects}

\subsection*{Weight space decompositions}

Fix a free $k$-module $V$ with basis $\{v_1,\ldots,v_n\}$.  For any
$X\in\Rep\Ga^d_k$ we describe a decomposition of $X(V)$ into weight
spaces; see also \cite[Corollary~2.12]{FS1997} for this decomposition
and a different argument.

The canonical decomposition  \eqref{eq:decomp}
\begin{equation*}
\Ga^{d,V}=\bigoplus_{\mu\in\La(n,d)}\Ga^{\mu}.
\end{equation*}
induces via the Yoneda isomorphism
$\Hom_{\Ga^d_k}(\Ga^{d,V},X)\xto{\sim} X(V)$ 
a decomposition
\[X(V)=\bigoplus_{\mu\in\La(n,d)}X(V)_{\mu}\qquad\text{with}\qquad
\Hom_{\Ga^d_k}(\Ga^\mu,X)\xto{\sim}X(V)_\mu.\] For each
$\mu\in\La(n,d)$ this isomorphism can be written as composition of
\[\Hom_{\Ga^d_k}(\Ga^\mu,X)\stackrel{\sim}\lto\Hom_{S_k(n,d)}(\Ga^\mu(V),X(V)),\quad
\p\mapsto \p_V\]
and
\[\Hom_{S_k(n,d)}(\Ga^\mu(V),X(V))\stackrel{\sim}\lto
X(V)_\mu,\quad\psi\mapsto \psi(v_1^{\otimes\mu_1}\otimes\cdots\otimes
v_n^{\otimes\mu_n}).\] Here, we identify
$\End_{\Ga^d\P_k}(V)=S_k(n,d)$ and note that
$v_1^{\otimes\mu_1}\otimes\cdots\otimes v_n^{\otimes\mu_n}$ generates
$\Ga^\mu(V)$ as $S_k(n,d)$-module.

The following lemma summarises this discussion. 

\begin{lem}\label{le:mor}
  \pushQED{\qed} Let $\mu\in\La(n,d)$ and set $V=k^n$.  For
  $X\in\Rep\Ga^d_k$ there  are natural isomorphisms
\[\Hom_{\Ga^d_k}(\Ga^\mu,X)\xto{\sim}\Hom_{S_k(n,d)}(\Ga^\mu(V),X(V))\xto{\sim}
X(V)_\mu.\qedhere\]
\end{lem}

We observe that the duality preserves the weight space decomposition.

\begin{lem}\label{le:duality}
 Let $\mu\in\La(n,d)$ and set $V=k^n$. For $X\in\Rep\Ga^d_k$ there is a natural isomorphism
\[ X^\circ(V)_\mu \cong X(V^*)_\mu^*.\]
\end{lem}
\begin{proof}
We have
  \[\Hom_{\Ga^d_k}(\Ga^{d,V},X^\circ)\cong X^\circ(V)=X(V^*)^*\cong\Hom_{\Ga^d_k}(\Ga^{d,V^*},X)^*.
\]
Now use Lemma~\ref{le:mor} and the canonical decomposition  
\begin{equation*}
\Ga^{d,V}\cong\bigoplus_{\mu\in\La(n,d)}\Ga^{\mu}\cong \Ga^{d,V^*}.\qedhere
\end{equation*}
\end{proof}

\subsection*{Standard morphisms}

We compute the weight spaces for $\Ga^\la$ and $S^\la$.  

Let $\la=(\la_1,\la_2,\ldots)$ and $\mu=(\mu_1,\mu_2,\ldots)$ be
sequences of non-negative integers satisfying
$\sum\la_i=d=\sum\mu_j$. Given a matrix $A=(a_{ij})_{i,j\ge 1}$ of non-negative
integers with $\la_i=\sum_j a_{ij}$ and $\mu_j=\sum_i a_{ij}$ for all
$i,j$, there is a \emph{standard morphism}
\[\g_A\colon\;\;\Ga^\mu=\bigotimes_j\Ga^{\mu_j}\lto
\bigotimes_j\Big(\bigotimes_i\Ga^{a_{ij}}\Big)
=\bigotimes_i\Big(\bigotimes_j\Ga^{a_{ij}}\Big)\lto \bigotimes_i\Ga^{\la_i}=\Ga^\la
\]
where the first morphism is the tensor product of the natural
inclusions $\Ga^{\mu_j}\to\bigotimes_i\Ga^{a_{ij}}$ and the second
morphism is the tensor product of the natural product maps
$\bigotimes_j\Ga^{a_{ij}}\to\Ga^{\la_i}$, as given by
\eqref{eq:comult}. Analogously, there is a morphism
\[\s_A\colon\;\;\Ga^\mu=\bigotimes_j\Ga^{\mu_j}\lto
\bigotimes_j\Big(\bigotimes_i T^{a_{ij}}\Big)
=\bigotimes_i\Big(\bigotimes_j T^{a_{ij}}\Big)\lto \bigotimes_i
S^{\la_i}= S^\la
\]
where $T^r=\otimes^r$ for any non-negative integer $r$, the first
morphism is the tensor product of the natural inclusions
$\Ga^{\mu_j}\to\bigotimes_i T^{a_{ij}}$, and the second morphism is
the tensor product of the natural product maps $\bigotimes_j
T^{a_{ij}}\to S^{\la_i}$.

\begin{lem}[{\cite[p.~8]{To1997}}]\label{le:totaro}
  Let $\la=(\la_1,\la_2,\ldots)$ and $\mu=(\mu_1,\mu_2,\ldots)$ be
  sequences of non-negative integers with $\sum\la_i=d=\sum\mu_i$.
\begin{enumerate}
\item The morphisms $\g_A$ form a $k$-basis of 
  $\Hom_{\Ga^d_k}(\Ga^\mu,\Ga^\la)$.\footnote{This  yields a basis of the Schur algebra 
$S_k(n,d)\cong\bigoplus_{\la,\mu\in\La(n,d)}\Hom_{\Ga^d_k}(\Ga^\mu,\Ga^\la)$.}
\item The morphisms $\s_A$ form a $k$-basis of 
  $\Hom_{\Ga^d_k}(\Ga^\mu, S^\la)$.
\end{enumerate} 
\end{lem}
\begin{proof}
  We may assume that $\la,\mu\in\La(n,d)$ and apply
  Lemma~\ref{le:mor}. Fix a free $k$-module $V$ with basis
  $\{v_1,\ldots,v_n\}$. Then we have an isomorphism
\[\Hom_{\Ga^d_k}(\Ga^\mu,\Ga^\la)\xto{\sim}
\Hom_{S_k(n,d)}(\Ga^\mu V,\Ga^\la V)\xto{\sim}(\Ga^\la V)_\mu.\] A
standard morphism $\g_A$ evaluated at $V$ takes the element
$v_1^{\otimes\mu_1}\otimes\cdots\otimes v_n^{\otimes\mu_n}$ to
$v_A=v_{\a^1}\otimes\cdots\otimes v_{\a^n}$ with $\a^i\in\La(n,\la_i)$
and $\a^i_j=a_{ij}$.  Now the assertion of part (1) follows from the
fact that the elements $v_A$ form a basis of $\Ga^\la V$ as $\mu$ runs
through $\La(n,d)$; cf.\ Example~\ref{ex:weyl}.

The proof of part (2) is analogous.
\end{proof}

For example, let $\la= (5,3,3,2)$ and $\mu=(1,3,3,2,2,2)$. For
\[A=\smatrix{1&2&2&0&0&0\\
0&1&1&0&1&0\\
0&0&0&2&0&1\\
0&0&0&0&1&1
}
\]
the morphism $\g_A$ evaluated at $V=k^6$ takes
$v_1^{\otimes\mu_1}\otimes\cdots\otimes v_6^{\otimes\mu_6}$ to the
  element 
  \[(v_1 ( v_2 \otimes v_2)( v_3\otimes v_3))\otimes (v_2 v_3
  v_5)\otimes ((v_4\otimes v_4) v_6)\otimes (v_5 v_6).\] 

\begin{exm}  
The special case $\la=(1,\ldots,1)=\mu$ yields the
  isomorphism \[\End_{\Ga^d_k}(\Ga^{(1,\ldots,1)})\cong k\mathfrak
  S_d.\]
\end{exm}

Let $\la$ be a partition and $T$ a filling of the corresponding Young
diagram. The \emph{content} of $T$ is by definition the sequence
$\mu=(\mu_1,\mu_2,\ldots)$ such that $\mu_i$ equals the number of
times the integer $i$ occurs in $T$.

\begin{exm}\label{ex:weyl}
  Let $\la=(\la_1,\ldots,\la_n)$ be a partition and set $V=k^n$. For a
  filling $T$ of the corresponding Young diagram with entries in
  $\{1,\ldots,n\}$, the element $v_T$ belongs to $(\Ga^\la V)_\mu$
  where $\mu$ equals the content of $T$.  The standard morphism
  $\g_A\colon\Ga^\mu\to\Ga^\la$ given by
  $a_{ij}=\card\{t\mid T(i,t)=j\}$ and evaluated at $V$ sends
  $v_1^{\otimes\mu_1}\otimes\cdots\otimes v_n^{\otimes\mu_n}$ to
  $v_T$.  If $\mu=\la$, then $T$ is the unique Young tableau such that
  all boxes of  the $i$th row have entry $i$.
\end{exm}

\subsection*{The dominance order}
We consider the \emph{dominance order} on the set of partitions of
weight $d$. Thus $\mu\le \la$ if $\sum_{i=1}^r\mu_i\le
\sum_{i=1}^r\la_i$ for all integers $r\ge 1$.

The following simple lemma explains the relevance of Young tableaux.

\begin{lem}\label{le:young}
  \pushQED{\qed} Let $\la$ and $\mu$ be partitions. Then there exists
  a Young tableau of shape $\la$ with content $\mu$ if and only if
  $\mu\le \la$.\qedhere
\end{lem}

The next proposition describes the weight spaces for Schur and Weyl
functors.

\begin{prop}\label{pr:weyl}
  Let $\la$ and $\mu$ be partitions of weight $d$.
\begin{enumerate}
\item $\Hom_{\Ga^d_k}(\Ga^\mu,W_\la)\neq 0$ if and only if $\mu\le\la$. Moreover,
  $\Hom_{\Ga^d_k}(\Ga^\la,W_\la)\cong k$. 
\item $\Hom_{\Ga^d_k}(\Ga^\mu,S_\la)\neq 0$ if and only if $\mu\le\la$. Moreover,
  $\Hom_{\Ga^d_k}(\Ga^\la,S_\la)\cong k$. 
\end{enumerate}
\end{prop}
\begin{proof}
  We apply Lemma~\ref{le:mor}.  The assertion for $W_\la$ then follows
  from the computation in Example~\ref{ex:weyl} and
  Lemma~\ref{le:young}, using the basis of a Weyl module from
  Proposition~\ref{pr:basis}. For $S_\la$ the assertion follows from
  the first part since $S_\la\cong W^\circ_\la$, using Lemma~\ref{le:duality}.
\end{proof}

\subsection*{Standard objects}

Let $\la$ be a partition of weight $d$. For  $X\in\Rep\Ga^d_k$ and any
partition $\mu$ of weight $d$  we
define the \emph{trace}
\[\tr_\mu X=\sum_{\p\colon\Ga^\mu\to X}\Im\p.\]
The \emph{standard object} corresponding to $\la$ is by definition
\[\Delta(\la)=\Ga^\la/\big(\sum_{\mu\not\le\la}\tr_\mu\Ga^\la\big)\]
where $\mu$ runs though all partitions of weight $d$. This yields an
exact sequence
\begin{equation}\label{eq:weyl}
0\lto \sum_{\mu\not\le\la}\tr_\mu\Ga^\la\lto \Ga^\la\lto\Delta(\la)\lto 0.
\end{equation}

\begin{thm}\label{th:weight}
  Let $\la$ be a partition of weight $d$.  The canonical morphism
  $\Ga^\la\to\Delta(\la)$ induces isomorphisms
\[W_\la\xto{\sim} \Delta(\la)\quad\text{and}\quad
\Hom_{\Ga^d_k}(\Delta(\la),\Delta(\la))\xto{\sim}\Hom_{\Ga^d_k}(\Ga^\la,\Delta(\la))
\xto{\sim} k.\]
\end{thm}
\begin{proof}
 The proof of \cite[Theorem~ II.3.16]{ABW1982} (which amounts to a
 proof of Proposition~\ref{pr:basis})
  shows that the functor $W_\la$ admits a presentation
\begin{equation}\label{eq:presentation}
  \bigoplus_{i\ge
    1}\bigoplus_{t=1}^{\la_{i+1}}\Ga^{\la(i,t)}\stackrel{\a}\lto\Ga^\la\lto
  W_\la\lto 0
\end{equation}
 where
\[\la(i,t)=(\la_1,\ldots,\la_{i-1},\la_{i}+t,\la_{i+1}-t,\la_{i+2},\ldots)\]
and $\Ga^{\la(i,t)}\to\Ga^\la$ is the standard morphism $\g_A$ given by the matrix
\[A=\diag(\la_1,\la_2,\ldots)+t E_{i+1,i}-tE_{i+1,i+1}.\] 

On the other hand, the definition of $\Delta(\la)$ yields a presentation 
\[\bigoplus_{\Ga^\mu\to\Ga^\la}\Ga^\mu \stackrel{\b}\lto\Ga^\la\lto
\Delta(\la)\lto 0\] where $\Ga^\mu\to\Ga^\la$ runs through all
morphisms such that $\mu\not\le\la$.

The morphism $\a$ factors through $\b$, since $\la(i,t)\not\le\la$ for
all pairs $i,t$. Conversely, $\b$ factors though $\a$, since
$\Hom_{\Ga^d_k}(\Ga^\mu,W_\la)= 0$ for all $\mu\not\le\la$ by
Proposition~\ref{pr:weyl}, and each $\Ga^\mu$ is projective.  It
follows that the canonical morphism $\Ga^\la\to \Delta(\la)$ induces
an isomorphism $W_\la\xto{\sim} \Delta(\la)$.

For the other pair of isomorphisms apply
$\Hom_{\Ga^d_k}(-,\Delta(\la))$ to the exact sequence \eqref{eq:weyl}
and use again Proposition~\ref{pr:weyl}.
\end{proof}

\subsection*{Costandard objects} The duality yields an analogue of
Theorem~\ref{th:weight} for Schur functors. The \emph{costandard
  object} corresponding to a partition $\la$ is by definition
  \[\nabla(\la)=\bigcap_{\mu\not\le\la}\rej_\mu S^\la\quad\text{with}\quad \rej_\mu
  X=\bigcap_{\p\colon X\to S^\mu}\Ker\p.\]

\begin{cor}\label{co:weight}
  Let $\la$ be a partition of weight $d$.  The canonical morphism
  $\nabla(\la)\to S^\la$ induces isomorphisms
 \[\nabla(\la)\xto{\sim} S_\la\quad\text{and}\quad
  \Hom_{\Ga^d_k}(\nabla(\la),\nabla(\la))\xto{\sim}\Hom_{\Ga^d_k}(\nabla(\la),S^\la)
  \xto{\sim} k.\]
Moreover, the canonical morphism $\Ga^\la\to\Delta(\la)$ induces isomorphisms
\[
\Hom_{\Ga^d_k}(\Delta(\la),\nabla(\la))\xto{\sim}\Hom_{\Ga^d_k}(\Ga^\la,\nabla(\la))
\xto{\sim} k.\]
\end{cor}
\begin{proof}
  From the definition we have $\nabla(\la)\cong\Delta(\la)^\circ$
  since $S^\mu\cong (\Ga^\mu)^\circ$ for each partition $\mu$. Thus
  the first set of isomorphisms follows directly from
  Theorem~\ref{th:weight} by applying the duality.

  For the last pair of isomorphisms apply
  $\Hom_{\Ga^d_k}(-,\nabla(\la))$ to the exact sequence
  \eqref{eq:weyl} and use Proposition~\ref{pr:weyl}.
\end{proof}
\subsection*{Simple objects}

We describe the simple objects in $\Rep\Ga^d_k$ provided that $k$ is a
local ring. For a partition $\la$ of weight $d$, consider the
subobject
\[U(\la)=\sum_{\mu<\la}\tr_\mu \De(\la)
+\Big(\sum_{\p\colon\Ga^\la\to\De(\la)}\Im\p\Big)\subseteq\De(\la)\]
where $\p$ runs through all morphisms corresponding to non-invertible
elements in $\Hom_{\Ga^d_k}(\Ga^\la,\De(\la))\cong k$, and set
\[L(\la)=\Delta(\la)/U(\la).\]

\begin{prop}\label{pr:standard}
Suppose that $k$ is a local ring and fix a partition $\la$ of weight $d$.
Then the functor $U(\la)$ is the unique maximal subobject of
  $\Delta(\la)$  and $L(\la)$ is a simple object in $\Rep\Ga^d_k$.
\end{prop}
\begin{proof}
  Let $X\subseteq \Delta(\la)$ be a subobject.  If $\tr_\la X= 0$,
  then $X\subseteq U(\la)$. If $\tr_\la X\neq 0$, then there is a
  nonzero morphism $\p\colon \Ga^\la\to X\rightarrowtail \Delta(\la)$,
  which is an epimorphism if and only if $\p$ corresponds to an
  invertible element, by Theorem~\ref{th:weight}. This follows by
  restricting $\p$ to the weight space for $\la$. Thus $U(\la)$ is the
  unique maximal subobject of $\Delta(\la)$ and $L(\la)$ is simple.
\end{proof}

The duality maps the unique simple quotient of $\Delta(\la)$ to
the unique simple subobject of $\nabla(\la)$. Next we show that the socle
of $\nabla(\la)$ is isomorphic to $L(\la)$.

\begin{lem}\label{le:simple}
  Let $S\in \Rep\Ga^d_k$ be simple and
  $\la=\max\{\mu\mid\Hom_{\Ga^d_k}(\Ga^\mu,S)\neq 0\}$. Then $S\cong
  L(\la)$.
\end{lem}
\begin{proof}
  Choose a nonzero morphism $\Ga^\la\to S$. This factors through the
  canonical morphism $\Ga^\la\to\Delta(\la)$. Thus $S\cong L(\la)$ by
  Proposition~\ref{pr:standard}.
\end{proof}

\begin{prop}\label{pr:simple}
Let $\la$ be a partition. Then $L(\la)^\circ\cong L(\la)$.
\end{prop}
\begin{proof}
 The assertion follows from Lemma~\ref{le:simple} using  Lemma~\ref{le:duality}.
\end{proof}

\begin{cor}\label{co:simples}
  \pushQED{\qed} Suppose that $k$ is a local ring and let $\La$ denote
  the set of partitions of weight $d$.  Then $\{L(\la)\}_{\la\in\La}$
  is a representative set of simple objects in $\Rep\Ga^d_k$. \qedhere
\end{cor}

\section{The Cauchy decomposition}

The Cauchy decomposition formula for Schur functors
\cite{ABW1982,DEP1980,DRS1974} is the analogue of Cauchy's formula for
symmetric functions \cite{Ca1841}. More precisely, the term `Cauchy
decomposition' refers to a filtration of symmetric powers whose
associated graded object is a direct sum of Schur functors. One
obtains the formula for symmetric functions by passing in
characteristic zero from polynomial representations of general linear
groups to their characters.

Fix $V,W\in\P_k$. For any non-negative integer $r$ there is a unique map
\[\psi^r\colon \Ga^r V\otimes \Ga^r W\lto \Ga^r(V\otimes W)\]
making the following square commutative.
\[\begin{tikzcd}
  \Ga^r V\otimes \Ga^r
  W\arrow{r}{\psi^r}\arrow[tail]{d}&
  \Ga^r(V\otimes W)\arrow[tail]{d}\\
  V^{\otimes r}\otimes W^{\otimes r}\arrow{r}{\sim}& (V\otimes
  W)^{\otimes r}
\end{tikzcd}\]
Extend this map for a partition
$\la=(\la_1,\ldots,\la_n)$ of weight $d$ to a map \[\psi^\la \colon
\Ga^\la V\otimes \Ga^\la W\lto \Ga^d(V\otimes W)\] which is given
as composite
\begin{multline*}
  \Ga^\la V\otimes \Ga^\la W\stackrel{\sim}\lto(\Ga^{\la_1} V\otimes
  \Ga^{\la_1} W)\otimes \cdots\otimes (\Ga^{\la_n} V\otimes
  \Ga^{\la_n} W)
  \xto{\psi^{\la_1}\otimes\cdots\otimes\psi^{\la_n}}\\
  \Ga^{\la_1}(V\otimes W)\otimes\cdots\otimes \Ga^{\la_n}(V\otimes W)\lto
  \Ga^d(V\otimes W)
\end{multline*}
with the last map  given by multiplication.

\subsection*{The lexicographic order}

We consider the \emph{lexicographic order} on the set of partitions of
weight $d$. Thus $\mu\le \la$ if for any integer $r\ge 1$ we have
$\mu_r\le \la_r$ whenever $\mu_i=\la_i$ for all $i<r$.  For a
partition $\la$ let $\la^-$ denote its immediate predecessor and
$\la^+$ its immediate successor. Set $(1,\ldots,1)^-=-\infty$ and
$(d)^+=+\infty$.

Note that the dominance order implies the lexicographic order.

\subsection*{The Cauchy filtration}

The \emph{Cauchy filtration} is by definition the chain
\begin{equation}\label{eq:cauchy}
0=F_{+\infty}\subseteq F_{(d)}\subseteq  F_{(d-1,1)}\subseteq \ldots \subseteq
 F_{(2,1,\ldots,1)}\subseteq F_{(1,\ldots,1)}=\Ga^d(V\otimes W)
\end{equation}
where $F_\la=\sum_{\mu\ge\la}\Im\psi^{\mu}$.

The following result describes the factors of the Cauchy filtration;
it is the analogue of \cite[Theorem~III.1.4]{ABW1982} for the Cauchy
filtration of $S^d(V\otimes W)$.\footnote{The Cauchy filtration of
  $S^d(V\otimes W)$ is obtained from \eqref{eq:cauchy} by duality. The
  approach via maps $\La^\la V\otimes\La^\la W\to S^d(V\otimes W)$
  seems to be more complicated.}

\begin{thm}[{\cite[Theorem~III.2.9]{HaKu1992}}]\label{th:cauchy}
Let $V,W\in\P_k$ and fix a partition $\la$ of weight $d$. Then the
  morphism $\psi^{\la}\colon\Ga^{\la} V\otimes \Ga^{\la} W\to
  F_\la$ induces an isomorphism
\[\Delta(\la) V\otimes \Delta(\la) W\stackrel{\sim}\lto F_\la/F_{\la^+}\]
which is functorial in $V$ and $W$ (with respect to morphisms in $\Ga^d\P_k$).
\end{thm}
\begin{proof}
  From the presentation \eqref{eq:presentation} of
  $W_\la\cong\Delta(\la)$ we deduce that there is a morphism
  $\bar\psi^\la$ making the following square commutative.
\[\begin{tikzcd}
\Ga^\la V\otimes \Ga^\la W\arrow{r}{\psi^\la}\arrow[two heads]{d}& F_\la\arrow[two heads]{d}\\
\Delta(\la) V\otimes \Delta(\la) W\arrow{r}{\bar\psi^\la} &F_\la/F_{\la^+}
\end{tikzcd}\] More precisely, consider the standard morphism
$\g_A\colon \Ga^{\la(i,t)}\to\Ga^\la$ arising in
\eqref{eq:presentation}.  The composition
\[\Ga^{\la(i,t)}V \otimes \Ga^{\la(i,t)}W\xto{\g_A V\otimes\g_A W}
\Ga^\la V\otimes \Ga^\la W\xto{\psi^\la} \Ga^d(V\otimes W)\] equals
a multiple of $\psi^{\la(i,t)}$, and we have $\Im\psi^{\la(i,t)}\subseteq F_{\la^+}$
since $\la(i,t) >\la$. This yields $\bar\psi^\la$. A computation of ranks (as in the proof
of \cite[Theorem~III.1.4]{ABW1982}) shows that $\bar\psi^\la$ is an
isomorphism.
\end{proof}

\begin{cor}\label{co:cauchy}
  Let $V\in\P_k$. There is a filtration
\[
0= X_{+\infty}\subseteq X_{(d)}\subseteq  X_{(d-1,1)}\subseteq \ldots \subseteq
 X_{(2,1,\ldots,1)}\subseteq X_{(1,\ldots,1)}=\Ga^d\Hom(V,-)
\]
such that  for each partition $\la$ of weight $d$
\[X_\la/X_{\la^+}\cong \De(\la)(V^*)\otimes\De(\la).\]
\end{cor}
\begin{proof}
  The filtration of $\Ga^d\Hom(V,-)\cong \Ga^d(V^*\otimes -)$ is given
  by the filtration \eqref{eq:cauchy}, replacing $V$ by $V^*$ and
  using its functoriality in $W$. Thus the description of
  $X_\la/X_{\la^+}$ follows from Theorem~\ref{th:cauchy}.
\end{proof}

  The filtration of $\Ga^d\Hom(V,-)$ induces a filtration for each direct
  summand of $\Ga^d\Hom(V,-)$. This follows from the functoriality of
  the filtration \eqref{eq:cauchy} in $V$.  The canonical
  isomorphism \[\End_{\Ga^d\P_k}(V)^\op\stackrel\sim\lto\End_{\Ga^d_k}(\Ga^d\Hom(V,-))\]
  then shows that a decomposition of $\Ga^d\Hom(V,-)$ yields a
  decomposition of each factor $X_\la/X_{\la^+}$.

\begin{cor}\label{co:cauchy-lambda}
Let $\mu$ be a partition of weight $d$. There is a filtration 
\[
0= Y_{+\infty}\subseteq Y_{(d)}\subseteq  Y_{(d-1,1)}\subseteq \ldots \subseteq
 Y_{\mu^+}\subseteq Y_\mu=\Ga^\mu
\]
such that  for each partition $\la \ge\mu$
\[Y_\la/Y_{\la^+} \cong \De (\la)^{K_{\la\mu}}\]
where $K_{\la\mu}$ equals the  number of Young tableaux of shape $\la$ and
content $\mu$.
\end{cor}
\begin{proof}
Let $\mu\in\La(n,d)$.  The functor $\Ga^\mu$ is a direct summand of $\Ga^d\Hom(k^n,-)$ and
  the functoriality of the filtration \eqref{eq:cauchy} in $V$ yields
  the filtration of $\Ga^\mu$ by passing for each partition $\la$ from
  $X_\la\subseteq \Ga^d\Hom(k^n,-)$ to the direct summand
  $Y_\la\subseteq \Ga^\mu$ corresponding to $\mu$. It follows from
  Corollary~\ref{co:cauchy} that for each partition $\la$
  \[Y_\la/Y_{\la^+}\cong\De(\la)(k^n)_\mu\otimes\De(\la),\]
  where $\De(\la)(k^n)_\mu$ is the weight space corresponding to
  $\mu$. This weight space can be computed and is nonzero if and only
  if $\la\ge \mu$ with respect to the dominance order, by
  Proposition~\ref{pr:weyl}. The precise description follows from the
  computation in Example~\ref{ex:weyl}, using the basis of a Weyl
  module from Proposition~\ref{pr:basis}.
\end{proof}

\begin{rem}
The number $K_{\la\mu}$ is called \emph{Kostka number}.
\end{rem}

The following is the analogue of Corollary~\ref{co:cauchy} for $S^d(V\otimes -)$.

\begin{cor}
  Let $V\in\P_k$.  There is a filtration
\[
0= Z_{-\infty}\subseteq Z_{(1,\ldots,1)}\subseteq  Z_{(2,1,\ldots,1)}\subseteq \ldots \subseteq
 Z_{(d-1,1)}\subseteq Z_{(d)}=S^d(V\otimes-)
\]
such that  for each partition $\la$ of weight $d$
\[Z_\la/Z_{\la^-}\cong \nabla(\la)V\otimes\nabla(\la).\]
\end{cor}
\begin{proof}
We modify the filtration of $ \Ga^d\Hom(V,-)$ from
Corollary~\ref{co:cauchy} as follows. Let $Z_\la$ denote the kernel of the
  epimorphism \[S^d(V\otimes -)\xto{\sim}
  \Ga^d\Hom(V,-)^\circ\twoheadrightarrow X_{\la^+}^\circ.\] Then we have
  \[Z_{\la}/Z_{\la^-}\cong (X_\la/X_{\la^+})^\circ\]
  which is a direct sum of copies $\De(\la)^\circ\cong \nabla(\la)$.
\end{proof}

\section{Highest weight structure}

Highest weight categories were introduced by Cline, Parshall, and
Scott \cite{CPS1988}. For the definition of a $k$-linear highest
weight category over a commutative ring $k$, see Definition~\ref{de:khwt}.

We fix a commutative ring $k$ and an integer $d\ge 0$. Let
$\rep\Ga^d_k$ denote the category of $k$-linear functors
$\Ga^d\P_k\to\P_k$, where $\P_k$ denotes the category of finitely
generated projective $k$-modules. Note that evaluation at $k^n$ gives
an equivalence $\rep\Ga^d_k\xto{\sim}\mod (S_k(n,d)^\op,k)$ for all
$n\ge d$, where $S_k(n,d)$ denotes the Schur algebra; cf.\ Example~\ref{ex:green}.

\begin{thm} 
  The category $\rep\Ga^d_k$ is a $k$-linear highest weight category
  with respect to the lexicographically ordered set of partitions of
  weight $d$. Thus there are exact sequences
\begin{equation*}
  0\lto U(\la)\lto P(\la)\lto \De(\la) \lto 0\qquad (\la\text{ a
    partition})
\end{equation*}
  in $\rep\Ga^d_k$ satisfying the following:
\begin{enumerate}
\item $\End_{\Ga^d_k}(\De(\la))\cong k$ for all $\la$.
\item $\Hom_{\Ga^d_k}(\De(\la),\De(\mu))=0$ for all $\la> \mu$. 
\item $U(\la)$ belongs to $\Filt\{\De(\mu)\mid \mu>\la\}$ for all $\la$.
\item $\bigoplus_\la P(\la)$ is a projective generator of $\rep\Ga^d_k$.
\end{enumerate}
\end{thm}
\begin{proof}
  Fix a partition $\la$ of weight $d$. We set $P(\la)=\Ga^\la$ and the
  canonical epimorphism $\Ga^\la\to\De(\la)$ yields the defining exact
  sequence, where $U(\la)=\sum_{\mu \not\le\la}\tr_\mu\Ga^\la$ (using
  the dominance order).  This gives (2) because every morphism
  $\De(\la)\to\De(\mu)$ lifts to a morphism $\Ga^\la\to\Ga^\mu$. More
  precisely, $\la>\mu$ (lexicographic order) implies $\la\not\le\mu$
  (dominance order), and therefore $\Ga^\la\to\Ga^\mu$ factors through
  $U(\mu)$.  (1) follows from Theorem~\ref{th:weight}, and (3) follows
  from Corollary~\ref{co:cauchy-lambda}. The canonical decomposition
  \eqref{eq:decomp} of each representable functor into summands of the
  form $\Ga^\la$ implies (4), since the representable functors form a
  set of projective generators of $\rep\Ga^d_k$.
\end{proof}

The module category of an algebra $A$ is a highest weight category if
and only if the algebra $A$ is quasi-hereditary \cite{CPS1988,
  Kr2015}.  Thus the equivalence \eqref{eq:schur} between
$\Rep\Ga^d_k$ and the category of modules over the Schur algebra
$S_k(n,d)$ for $n\ge d$ yields the following (see \cite[\S7]{Gr1993}
for historical comments).

\begin{cor}
\pushQED{\qed}
The Schur algebra $S_k(n,d)$ is quasi-hereditary for all $n\geq
d$.\qedhere
\end{cor}

\section{Characteristic tilting objects and Ringel duality}

Fix a commutative ring $k$ and an integer $d\ge 0$. We describe the
characteristic tilting object for the highest weight category
$\rep\Ga^d_k$ and show that $\rep\Ga^d_k$ is Ringel self-dual. These
results are due to Donkin \cite{Do1993} when $k$ is a field.

We begin with some preparations and recall the following result.

\begin{prop}[{\cite[Theorem~3.7]{Bo1988}}]\label{pr:boffi}
Let $X\in\Filt(\nabla)\subseteq\rep\Ga^d_k$ and
$Y\in\Filt(\nabla)\subseteq\rep\Ga^e_k$. Then $X\otimes Y$ is in
$\Filt(\nabla)\subseteq\rep\Ga^{d+e}_k$.\qed
\end{prop}

\begin{prop}\label{pr:exterior-tilting}
Let $\la$ be a partition of weight $d$. Then $\La^\la$ is in
$\Filt(\De)\cap\Filt(\nabla)$.
\end{prop}
\begin{proof}
  We have $\La^{\la_i}=S_{(1,\ldots,1)}\in \Filt(\nabla)$ for all
  $i$. Thus $\La^{\la}\in \Filt(\nabla)$ by
  Proposition~\ref{pr:boffi}.  Analogously,
  $\La^\la\cong(\La^\la)^\circ\in\Filt(\nabla)^\circ=\Filt(\De)$.
\end{proof}

Next recall from \cite{Kr2013} that there is an adjoint pair of functors
\[\La^d\otimes_{\Ga^d_k}-\colon \Rep\Ga^d_k\lto
\Rep\Ga^d_k\quad\text{and}\quad \HOM_{\Ga^d_k}(\La^d,-)\colon \Rep\Ga^d_k\lto \Rep\Ga^d_k.\]

\begin{prop}[{\cite[Corollary~3.8]{Kr2013}}]\label{pr:tensor}
\pushQED{\qed}
The functor $\La^d\otimes_{\Ga^d_k}-$ maps $\Ga^\la$ to $\La^\la$
and  induces an equivalence
\[\add\{\Ga^\la\mid\la\text{ partition of }d\}\xto{\ \sim\ }
\add\{\La^\la\mid\la\text{ partition of }d\}.\qedhere\]

\end{prop}

\begin{prop}\label{pr:delta-to-nabla}
  Let $\la$ be a partition of weight $d$ and $\la'$ its conjugate
  partition. The functor $\La^d\otimes_{\Ga^d_k}-$ maps $W_\la$ to
  $S_{\la'}$.
\end{prop}
\begin{proof}
  We use the presentation \eqref{eq:presentation} of $W_\la$, and the
  functor $\La^d\otimes_{\Ga^d_k}-$ maps this to the following exact
  sequence.
\begin{equation*}
  \bigoplus_{i\ge
    1}\bigoplus_{t=1}^{\la_{i+1}}\La^{\la(i,t)}\lto\La^\la\lto
 \La^d\otimes_{\Ga^d_k} W_\la\lto 0
\end{equation*}
On the other hand, $S_{\la'}$ admits the presentation
\begin{equation*}
  \bigoplus_{i\ge
    1}\bigoplus_{t=1}^{\la_{i+1}}\La^{\la(i,t)}\stackrel{\b}\lto\La^\la\lto
  S_\la\lto 0
\end{equation*}
where $\b$ is the analogue of the morphism $\a$ in \eqref{eq:presentation} 
\cite[Theorem~ II.2.16]{ABW1982}. Thus the assertion follows.
\end{proof}

Let $\la$ be a partition of weight $d$ and $\mathbf\Ga(W_\la)$  a projective resolution of $W_\la$.  Then
the left derived functor of $\La^d\otimes_{\Ga^d_k}-$ evaluated at
$W_\la$ is given by the homology of
$\La^d\otimes_{\Ga^d_k}\mathbf\Ga(W_\la)$.

\begin{lem}\label{le:tor}
  We have $H_p(\La^d\lotimes_{\Ga^d_k}\mathbf\Ga(W_\la))=0$ for $p>0$.
\end{lem}
\begin{proof}
The objects $S^\mu$ form a set of injective cogenerators of $\Rep\Ga^d_k$.  Adjointness gives
\[\Hom_{\Ga^d_k}(\La^d\otimes_{\Ga^d_k}\mathbf\Ga(W_\la),S^\mu)\cong
\Hom_{\Ga^d_k}(\mathbf\Ga(W_\la),\HOM_{\Ga^d_k}(\La^d ,S^\mu)).\]
We have
\[\HOM_{\Ga^d_k}(\La^d,S^\mu)\cong(\La^d\otimes_{\Ga^d_k}\Ga^\mu)^\circ
\cong(\La^\mu)^\circ\cong\La^\mu\]
where the first isomorphism follows from \cite[Lemma~2.7]{Kr2013} and
the second uses Proposition~\ref{pr:tensor}. It remains to observe
that 
\[H_p(\Hom_{\Ga^d_k}(\mathbf\Ga(W_\la),\HOM_{\Ga^d_k}(\La^d
,S^\mu)))\cong\Ext^p_{\Ga^d_k}(W_\la,\La^\mu)\]
vanishes for $p>0$ by Corollary~\ref{co:Ext-orth} and
Proposition~\ref{pr:exterior-tilting}.
\end{proof}

\begin{thm}\label{th:char-tilting}
The functor $\La^d\otimes_{\Ga^d_k}-$
  induces an equivalence
\[\Filt\{\De(\la)\mid \la\text{ partition of }d\}\xto{\ \sim\ }
\Filt\{\nabla(\la)\mid \la\text{ partition of }d\}.\]
Therefore the highest weight category $\rep\Ga^d_k$ is Ringel
self-dual with characteristic tilting object $\bigoplus_\la\La^\la$.
\end{thm}
\begin{proof}
  We identify $\De(\la)=W_\la$ and $\nabla(\la)=S_\la$ for each
  partition $\la$; see Theorem~\ref{th:weight} and
  Corollary~\ref{co:weight}.

  The functor $\La^d\otimes_{\Ga^d_k}-$ maps $\De(\la)$ to
  $\nabla(\la')$ by Proposition~\ref{pr:delta-to-nabla}, and it is
  exact on $\Filt(\De)$ by Lemma~\ref{le:tor}.  Note that each object
  $\La^\la=\La^d\otimes_{\Ga^d_k}\Ga^\la$ is projective in
  $\Filt(\nabla)$ by Corollary~\ref{co:Ext-orth} and
  Proposition~\ref{pr:exterior-tilting}.  Thus the functor
  $\La^d\otimes_{\Ga^d_k}-$ maps the projective generators of
  $\Filt(\De)$ fully faithfully to projective generators of
  $\Filt(\nabla)$; see Proposition~\ref{pr:tensor}. This gives the
  equivalence.  The object $\bigoplus_\la\La^\la$ is a characteristic
  tilting object, because it is a projective generator of
  $\Filt(\nabla)$; see Corollary~\ref{co:tilting-characterisation}.
  The property of $\rep\Ga^d_k$ to be Ringel self-dual follows from
  Corollary~\ref{co:tilting}.
\end{proof}

\section*{Acknowledgements} 
I am grateful to Andrew Hubery for many helpful comments on this
subject. Additional thanks are due to Bernard Leclerc for pointing me
to \cite{Ca1841}.


\begin{thebibliography}{99}
%
\bibitem{ABW1982} K. Akin, D. A. Buchsbaum\ and\ J. Weyman, Schur functors and Schur
complexes, Adv. in Math. {\bf 44} (1982), no.~3, 207--278.
%
\bibitem{Au1974} M. Auslander, Representation theory of Artin
algebras. I, Comm. Algebra {\bf 1} (1974), 177--268.
%
\bibitem{BBD1982} A. A. Be\u\i linson, J. Bernstein\ and\ P. Deligne,
  Faisceaux pervers, in {\it Analysis and topology on singular spaces,
    I (Luminy, 1981)}, 5--171, Ast\'erisque, 100, Soc. Math. France,
  Paris, 1982.
%
\bibitem{Bo1988} G. Boffi, The universal form of the
  Littlewood-Richardson rule, Adv. in Math. {\bf 68} (1988), no.~1,
  40--63.
%
\bibitem{Bo1981} N. Bourbaki, {\it \'El\'ements de math\'ematique.
    Alg\`ebre. Chapitres 4 \`a 7}, Lecture Notes in Mathematics, 864,
  Masson, Paris, 1981.
%
\bibitem{Ca1841} A.-L. Cauchy, M\'emoire sur les fonctions altern\'ees
  et sur les sommes altern\'ees, Exercices d'analyse et de
  phys. math., ii (1841), 151--159; or {\OE}uvres compl\`etes, 2\`eme
  s\'erie xii, Gauthier-Villars, Paris, 1916, 173--182.
%
\bibitem{CPS1988} E. Cline, B. Parshall\ and\ L. Scott,
  Finite-dimensional algebras and highest weight categories, J. Reine
  Angew. Math. {\bf 391} (1988), 85--99. 
%
\bibitem{CPS1990} E. Cline, B. Parshall\ and\ L. Scott, Integral and
  graded quasi-hereditary algebras. I, J. Algebra {\bf 131} (1990),
  no.~1, 126--160.
%
\bibitem{DEP1980} C. de Concini, D. Eisenbud\ and\ C. Procesi, Young diagrams and
determinantal varieties, Invent. Math. {\bf 56} (1980), no.~2,
129--165. 
%
\bibitem{DR1992} V. Dlab\ and\ C. M. Ringel, The module theoretical
  approach to quasi-hereditary algebras, in {\it Representations of
    algebras and related topics (Kyoto, 1990)}, 200--224, London
  Math. Soc. Lecture Note Ser., 168, Cambridge Univ. Press, Cambridge,
  1992.
%
\bibitem{Do1987} S. Donkin, On Schur algebras and related algebras. II, J. Algebra {\bf
  111} (1987), no.~2, 354--364.
%
\bibitem{Do1993} S. Donkin, On tilting modules for algebraic groups,
  Math. Z. {\bf 212} (1993), no.~1, 39--60.
%
\bibitem{DRS1974} P. Doubilet, G.-C. Rota\ and\ J. Stein, On the
  foundations of combinatorial theory. IX. Combinatorial methods in
  invariant theory, Studies in Appl. Math. {\bf 53} (1974), 185--216.
%
\bibitem{DS1994} J. Du\ and\ L. Scott, Lusztig conjectures, old and
  new. I, J. Reine Angew. Math. {\bf 455} (1994), 141--182.
%
\bibitem{Ef2014} A. I. Efimov, Derived categories of Grassmannians
  over integers and modular representation theory, arXiv:1410.7462.
%
\bibitem{FS1997} E. M. Friedlander\ and\ A. Suslin, Cohomology of
  finite group schemes over a field, Invent. Math. {\bf 127} (1997),
  no.~2, 209--270.
%
\bibitem{Fu1997} W. Fulton, {\it Young tableaux}, London Mathematical
  Society Student Texts, 35, Cambridge Univ. Press, Cambridge, 1997.
%
\bibitem{Gr1980} J. A. Green, {\it Polynomial representations of ${\rm
      GL}\sb{n}$}, Lecture Notes in Mathematics, 830, Springer,
  Berlin, 1980.
%
\bibitem{Gr1993} J. A. Green, Combinatorics and the Schur algebra,
  J. Pure Appl. Algebra {\bf 88} (1993), no.~1-3, 89--106.
%
\bibitem{Ha2012} M. Hashimoto, Schur algebras, Lecture notes,
  \spverb|www.math.nagoya-u.ac.jp/~hasimoto/| (2012).
%
\bibitem{HaKu1992} M. Hashimoto\ and\ K. Kurano, Resolutions of
  determinantal ideals: $n$-minors of $(n+2)$-square matrices,
  Adv. Math. {\bf 94} (1992), no.~1, 1--66.
%
\bibitem{Ka2015} M. Kalck, Derived categories of quasi-hereditary
  algebras and their derived composition series, this volume.
%
\bibitem{Ke1990} B. Keller, Chain complexes and stable categories, Manuscripta
Math. {\bf 67} (1990), no.~4, 379--417.
%
\bibitem{KKO2014} S. Koenig, J. K\"ulshammer\ and\ S. Ovsienko,
  Quasi-hereditary algebras, exact Borel subalgebras, $A\sb
  \infty$-categories and boxes, Adv. Math. {\bf 262} (2014), 546--592.
%
\bibitem{Kr2013} H. Krause, Koszul, Ringel and Serre duality for
  strict polynomial functors, Compos. Math. {\bf 149} (2013), no.~6,
  996--1018.
%
\bibitem{Kr2015} H. Krause, Highest weight categories and
recollements, preprint 2015, \texttt{arXiv:1506.01485}.
%
\bibitem{Ku1998} N. J. Kuhn, Rational cohomology and cohomological
  stability in generic representation theory, Amer. J. Math. {\bf 120}
  (1998), no.~6, 1317--1341. 
%
\bibitem{Ku2015} J. K\"ulshammer, In the bocs seat:
Quasi-hereditary algebras and representation type, this volume.
%
\bibitem{La1977} A. Lascoux, Polyn\^omes sym\'etriques, foncteurs de Schur et 
grassmanniens, Th\`ese, Universit\'e Paris 7, 1977.
%
\bibitem{MS2008} V. Mazorchuk\ and\ C. Stroppel, Projective-injective
  modules, Serre functors and symmetric algebras, J. Reine
  Angew. Math. {\bf 616} (2008), 131--165.
%
\bibitem{PS1988} B. Parshall\ and\ L.  Scott, Derived categories,
  quasi-hereditary algebras, and algebraic groups, in {\it Proceedings
    of the Ottawa-Moosonee Workshop in Algebra}, Carleton Univ. Notes
  {\bf 3} (1988), 1--105.
%
\bibitem{Pi2003} T. Pirashvili, Introduction to functor homology, in {\it Rational
  representations, the Steenrod algebra and functor homology}, 1--26,
Panor. Synth\`eses, 16 Soc. Math. France, Paris, 2003.
%
\bibitem{Qu1973} D. Quillen, Higher algebraic $K$-theory. I, in {\it
    Algebraic $K$-theory, I: Higher $K$-theories (Proc. Conf.,
    Battelle Memorial Inst., Seattle, Wash., 1972)}, 85--147. Lecture
  Notes in Math., 341, Springer, Berlin, 1973.
%
\bibitem{Ri1991} C. M. Ringel, The category of modules with good
  filtrations over a quasi-hereditary algebra has almost split
  sequences, Math. Z. {\bf 208} (1991), no.~2, 209--223.
%
\bibitem{Ro2008} R. Rouquier, $q$-Schur algebras and complex
  reflection groups, Mosc. Math. J. {\bf 8} (2008), no.~1, 119--158.
%
\bibitem{To1997} B. Totaro, Projective resolutions of representations
  of GL(n), J. Reine Angew. Math. 482 (1997) 1–13.
%
\bibitem{Ka2012} W. van der Kallen, Nantes lectures on bifunctors and
  CFG, arXiv:1208.3097 (2012).
%
\bibitem{Sc1901} I. Schur, \"Uber eine Klasse von Matrizen, die sich
  einer gegebenen Matrix zuordnen lassen, Dissertation, Universit\"at
  Berlin, 1901. In I. Schur, Gesammelte Abhandlungen I, 1--70,
  Springer, Berlin, 1973.
%
\bibitem{Sc1987} L. L. Scott, Simulating algebraic geometry with
  algebra. I. The algebraic theory of derived categories, in {\it The
    Arcata Conference on Representations of Finite Groups (Arcata,
    Calif., 1986)}, 271--281, Proc. Sympos. Pure Math., 47, Part 2,
  Amer. Math. Soc., Providence, RI, 1987.
%
\bibitem{We2003} J. Weyman, {\it Cohomology of vector bundles and
    syzygies}, Cambridge Tracts in Mathematics, 149, Cambridge
  Univ. Press, Cambridge, 2003. 
%
\end{thebibliography}
\end{document}